\documentclass[10pt,oneside,reqno]{amsart}

  \usepackage[utf8]{inputenc}
  \usepackage[T1]{fontenc}
  \usepackage[russian,ngerman,english]{babel}
  \usepackage{lmodern}

  \usepackage{hyperref}
  \usepackage{amsmath}
  \usepackage{amsthm}
  \usepackage{amsfonts}
  \usepackage{amssymb}
  \usepackage{amscd}
  \usepackage{amsbsy}

  \usepackage{pdfpages}
  \usepackage{fancyhdr}
  \usepackage{geometry}
  \usepackage{setspace}
  \usepackage{xcolor}
  \usepackage{pifont}

  \usepackage{graphicx}
  \usepackage{tabularx}
  \usepackage{epsfig}
  \usepackage[all]{xy}
  \usepackage{tikz}
  \usetikzlibrary{matrix}

  \usepackage{fixmath}

  \usepackage{appendix}

  \usepackage{enumerate}

  \usepackage[sort, numbers]{natbib}
  \usepackage{bibentry}

  \newtheorem{theorem}{Theorem}[section]

  \newtheorem{lemma}[theorem]{Lemma}

  \newtheorem{corollary}[theorem]{Corollary}
  \newtheorem{hypothesis}[theorem]{Hypothesis}
  \theoremstyle{definition}

  \newtheorem*{remark}{Remark}

  \newtheorem{example}[theorem]{Example}

  \numberwithin{equation}{section}

  \newcommand{\E}{{\mathsf E}}

  \newcommand{\supp}{{\operatorname{supp}}}

  \newcommand{\tr}{\operatorname{tr}}

\DeclareMathOperator{\Int}{int}
\DeclareMathOperator{\Mat}{Mat}

  \newcommand{\cC}{{\mathcal{C}}}

  \newcommand{\cJ}{{\mathcal{J}}}
  
  \newcommand{\cK}{{\mathcal{K}}}

  \newcommand{\cS}{{\mathcal{S}}}
  \newcommand{\cT}{{\mathcal{T}}}

  \newcommand{\fH}{{\mathfrak{H}}}
  \newcommand{\fK}{{\mathfrak{K}}}
  \newcommand{\fM}{{\mathfrak{M}}}

  \newcommand{\fS}{{\mathfrak{S}}}

  \newcommand{\e}{{\varepsilon}}

  \newcommand{\z}{\zeta}

  \renewcommand{\Re}{\operatorname{Re}}
  \renewcommand{\Im}{\operatorname{Im}}

  \newcommand{\esssupp}{\operatorname{ess\,supp}}

  \newcommand{\SL}{\text{SL}}

  \renewcommand{\i}{\infty}

  \newcommand{\bbD}{\mathbb{D}}
  \newcommand{\bbN}{\mathbb{N}}
  \newcommand{\bbZ}{\mathbb{Z}}
  
  \newcommand{\bbR}{\mathbb{R}}
  \newcommand{\bbC}{\mathbb{C}}

  \newcommand{\AC}{\mathrm{AC}}
  \newcommand{\loc}{\mathrm{loc}}

  \newcommand{\ac}{\mathrm{ac}}

  \newcommand{\spann}{\mathrm{span}}

  \allowdisplaybreaks

\author[B.\ Eichinger]{Benjamin Eichinger}

\address{Institute for Analysis, Johannes Kepler University, Linz, A-4040, Austria}

\email{benjamin.eichinger@jku.at}

\thanks{B.\ E.\ was supported by the Austrian Science Fund FWF, project no: P33885}

\author[M.\ Luki\'c]{Milivoje Luki\'c}

\address{Department of Mathematics, Rice University, Houston, TX~77005, USA}

\email{ milivoje.lukic@rice.edu}

\thanks{M.L.\ was supported in part by NSF grant DMS--1700179.}

\author[B.\ Simanek]{Brian Simanek}

\address{Department of Mathematics, Baylor University, Waco, TX, USA}

\email{brian\_simanek@baylor.edu}

\thanks{B.\ S.\ was supported by Simons Foundation Collaboration grant 707882.}

\title{An approach to universality using Weyl $m$-functions}
\begin{document}
	\maketitle
	
	\begin{abstract}
		We describe an approach to universality limits for orthogonal polynomials on the real line which is completely local and uses only the boundary behavior of the Weyl $m$-function at the point. We show that bulk universality of the Christoffel--Darboux kernel holds for any point where the imaginary part of the $m$-function has a positive finite nontangential limit. This approach is based on studying a matrix version of the Christoffel--Darboux kernel and the realization that bulk universality for this kernel at a point is equivalent to the fact that the corresponding $m$-function has normal limits at the same point. Our approach automatically applies to other self-adjoint systems with $2\times 2$ transfer matrices such as continuum Schr\"odinger and Dirac operators. We also obtain analogous results for orthogonal polynomials on the unit circle. 
	\end{abstract}


\section{Introduction}

Let $\mu$ be a probability measure on $\bbR$ with all finite moments,
\[
\int \lvert \xi \rvert^n \,d\mu(\xi) < \infty, \qquad \forall n\in\bbN.
\]
We will always assume that $\mu$ has infinite support (in the sense of cardinality; this is not to be confused with boundedness of the support -- $\mu$ may or may not have compact support). Then associated orthonormal polynomials $\{p_j(z) \}_{j=0}^{\infty}$ are obtained by the Gram--Schmidt process from the sequence of monomials $\{z^j\}_{j=0}^\infty$ in $L^2(\bbR,d\mu)$.  The Christoffel--Darboux (CD) kernel is
\[
K_n(z,w)=\sum_{j=0}^{n-1} p_j(z)  \overline{p_j(w)}.
\]
It is the reproducing kernel for the subspace $\spann \{1, z,\dots, z^{n-1}\}$ in $L^2(\bbR,d\mu)$ and is a central object in the theory of orthogonal polynomials \cite{SimonCDkernel}. We note that there are alternative notations, with different placements of complex conjugate and/or summation up to $n$ instead of $n-1$.

Of particular interest are the rescaled limits of CD kernels
\begin{equation}\label{genuniv}
\lim_{n\to\infty} \frac 1{\tau_n} K_n \left(\xi+ \frac z{\tau_n} ,\xi+ \frac w{\tau_n} \right)
\end{equation}
for an appropriate sequence $\tau_n \to \infty$ and $z,w\in\bbC, \xi\in\bbR$. They are called universality limits because the limit is often found to be a standard kernel: the most common phenomenon is bulk universality, in which case \eqref{genuniv} gives a (suitably rescaled) $\sin(\overline{w} - z) / (\overline{w} - z)$, called the sinc (or sine) kernel. 
The motivation for this line of inquiry comes from random matrix theory, where universality limits describe the local eigenvalue statistics of unitary ensembles \cite{DeiftOPandRM}. The interest in this ``universal'' behavior dates back to foundational work in random matrix theory by Wigner \cite{Wigner55}, who proposed local eigenvalue statistics of random matrices as a model for local statistical behavior of resonances in scattering theory.

Universality limits were first studied at scale $\tau_n = K_n(\xi,\xi)$ and proved by Riemann--Hilbert methods for a large family of  purely absolutely continuous measures with real analytic weights \cite{DKMVZ1,DKMVZ2,DKMVZ3,DeiftOPandRM,KvL02}. A breakthrough result by Lubinsky \cite{LubinskyAnnals}, with further developments by \cite{Findley08,SimonTwoExt08,TotikUniv09,Totik16}, instead requires Stahl--Totik regularity \cite{StahlTotik92} (a global condition)  and Lebesgue point and local Szeg\H o conditions; most of those results are restricted to compactly supported measures (but see \cite{LevLub09,Mitkovski}). In this paper, we present a new approach to universality limits which is based on Weyl $m$-functions.

We will always assume that $\mu$ corresponds to a determinate moment problem, i.e., that there is no other measure on $\bbR$ with the same sequence of moments; one of the equivalent characterizations of this condition is that 
\[
\lim_{n\to\infty} K_n(z,z)  = \infty
\]
for one (and therefore all) $z\in \bbC \setminus \bbR$. This is equivalent to the limit point case in the language of Jacobi matrices \cite[Section 3.8]{SimonSzegoBook}. In particular, any compactly supported $\mu$ corresponds to a determinate moment problem; so do the normalized exponential weights $C_\alpha e^{-\lvert \xi \rvert^\alpha} \,d\xi$ for $\alpha \ge 1$ and $e^{-Q(\xi)} \,d\xi$ for even degree polynomials $Q$ with positive leading coefficient  \cite[Section 2.5]{DeiftOPandRM}. 
In summary, the assumptions on $\mu$ will be:

\begin{hypothesis} \label{hypothesismu}
$\mu$ is a probability measure on $\bbR$ with infinite support and finite moments.  It corresponds to a  determinate moment problem (i.e. the limit point case).
\end{hypothesis}

In our first theorem, we conclude a universality limit from a nontangential limit of the imaginary part of the $m$-function, defined on $\bbC_+ = \{ z \in \bbC \mid \Im z > 0\}$ by
\begin{equation}\label{Stieltjes}
m(z) = \int \frac 1{\xi - z} \, d\mu(\xi).
\end{equation}

\begin{theorem}\label{thmImaginaryLimitJacobi}
Let $\mu$ obey Hypothesis~\ref{hypothesismu}. Let $m(z)$ be its Stieltjes transform \eqref{Stieltjes}. Let $\xi \in\bbR$ and assume  that for some $0 < \alpha < \pi/2$, the limit
\begin{equation} \label{limitImm}
f_\mu(\xi) := \frac{1}{\pi}\lim\limits_{\substack{z \to \xi \\ \alpha \le \arg (z-\xi) \le \pi - \alpha}} \Im m(z)
\end{equation}
exists and $0 < f_\mu(\xi) < \infty$. Then uniformly on compact regions of $(z,w) \in \bbC \times \bbC$,
\begin{equation}\label{bulkuniversality3aug}
\lim_{n\to\infty}\frac{K_n\left(\xi+\frac{z}{f_\mu(\xi) K_n(\xi,\xi)},\xi+\frac{w}{f_\mu(\xi) K_n(\xi,\xi)}\right)}{K_n(\xi,\xi)}=\frac{\sin(\pi(\overline{w}-z))}{\pi(\overline{w}-z)}
\end{equation}
(of course, the right-hand side is interpreted as $1$ for $\overline{w} = z$).
\end{theorem}

\begin{remark} \label{firstremark}
\begin{enumerate}[(a)]
\item Lebesgue-almost everywhere on $\bbR$, the limit \eqref{limitImm} exists. The limit $f_\mu(\xi)$ recovers the a.c. part of the measure in the sense that
\begin{equation}\label{LebesgueDecomposition}
d\mu(\xi) = f_\mu(\xi) d\xi + d\mu_s(\xi)
\end{equation}
where $\mu_s$ is a measure singular with respect to Lebesgue measure. In particular, the set
\[
\Sigma_\ac(\mu) = \{ \xi \in\bbR \mid f_\mu(\xi) \text{ exists and }0 < f_\mu(\xi) < \infty\}
\]
is said to be an essential support for the a.c. spectrum.
\item The point $\xi \in \bbR$ is said to be a Lebesgue point of $\mu$ if
\begin{equation}\label{Lebesguemu}
\lim_{\epsilon \downarrow 0} \frac{ \int_{(\xi-\e,\xi+\e)} \lvert f_\mu(s) - f_\mu(\xi) \rvert ds }{2\e} = \lim_{\epsilon \downarrow 0} \frac{ \mu_s((\xi - \e, \xi + \e)) }{2\e} =  0
\end{equation}
(more precisely, $f_\mu$ is usually thought of as an equivalence class of functions in $L^1_\loc(\bbR,d\xi)$, and $\xi$ is said to be a Lebesgue point of $\mu$ if there is a choice of value $f_\mu(\xi)$ such that \eqref{Lebesguemu} holds). At any Lebesgue point of $\mu$, for any $0 < \alpha < \pi/2$, the limit \eqref{limitImm} exists \cite[Theorems 11.22, 11.23]{RudinRealandComplex} and the two pointwise meanings of $f_\mu(\xi)$ (from \eqref{limitImm} or \eqref{Lebesguemu}) match. In particular, Theorem~\ref{thmImaginaryLimitJacobi} applies at every Lebesgue point with $0 < f_\mu(\xi) < \infty$, so it applies at a.e. $\xi$ in the essential support of the a.c. spectrum. 
\item Condition \eqref{limitImm} is unaffected by a change in the measure away from some neighborhood of $\xi$, so it is a local condition on the measure.
\item In terms of the $m$-function, the condition \eqref{limitImm} is optimal; it cannot be replaced by a normal limit. There is a class of examples by Moreno--Martinez--Finkelstein--Sousa \cite{MorenoFinkelSousaConstr}  with a jump in the weight where the normal limit $\lim_{y\downarrow 0} \Im m(\xi + iy)$ exists but rescaled CD kernels converge to a different kernel. 
\end{enumerate}
\end{remark}

 Until now, the strongest results on convergence of rescaled kernels that were proven with local conditions on the measure come from Lubinsky’s work \cite{LubinskyJourdAnalyse12,LubinskyMMJ}, where his conclusions establish convergence in measure or convergence in mean.  It was not previously known whether the bulk universality limit \eqref{bulkuniversality3aug} can be concluded pointwise from only a local condition on the measure, and this is explicitly stated as an open problem in \cite{LubinskyMMJ}. Theorem~\ref{thmImaginaryLimitJacobi} resolves that open problem.  Many of the previous results assume a compactly supported measure; this does not matter for our approach and is another demonstration of its locality.

If the limit \eqref{genuniv} is to exist for some scale $\tau_n$, taking $z=w=0$ shows that the limit of $K_n(\xi,\xi) / \tau_n$ must exist. Conversely, if this limit is finite and nonzero, the locally uniform convergence allows one to change the scale from $K_n(\xi,\xi)$ to $\tau_n$. The growth rate of $K_n(\xi,\xi)$ depends on global properties of the measure. For weights of the form $e^{-Q(\xi)} \,d\xi$ with polynomials $Q$, $K_n(\xi,\xi)$ has power law behavior with exponent dependent on the degree of $Q$ \cite{DeiftOPandRM}. For compactly supported measures, the linear behavior 
\begin{equation} \label{linearscale1}
\lim_{n\to\infty} \frac{ K_n(\xi,\xi)} n = \frac{f_\E(\xi)}{f_\mu(\xi)}
\end{equation}
(where $f_\E$ denotes the density of the equilibrium measure of the essential spectrum $\E = \esssupp \mu$) is heuristically natural, but in order to prove it, some global assumption on $\mu$ is needed. M\'at\'e--Nevai--Totik \cite{MNT} proved \eqref{linearscale1} for Stahl--Totik regular measures \cite{StahlTotik92} (a global condition) under the local assumptions that $f_\mu(\xi) > 0$, $\log f_\mu$ is integrable in a neighborhood of $\xi$, and $\xi$ is a Lebesgue point of both the measure $\mu$ and the function $\log f_\mu$, for the case $\E = [-2,2]$; Totik \cite{TotikJAM} generalized this to arbitrary compacts $\E \subset \bbR$. Of course, Theorem~\ref{thmImaginaryLimitJacobi} can be combined with all such statements.  In fact, it is now  clear that bulk universality at scale $K_n(\xi,\xi)$ and the linear behavior \eqref{linearscale1} are two separate questions which can be combined in their shared scope of applicability.

The approach to universality by localization and smoothing  relies on a combination of Stahl--Totik regularity and local assumptions on the measure at $\xi$. Initially this was proved by Lubinsky~\cite{LubinskyAnnals} for $\xi \notin \supp \mu_s$ with $f_\mu$ continuous, strictly positive at $\xi$; this was generalized in various ways by Simon~\cite{SimonTwoExt08}, Totik~\cite{TotikUniv09}, and Findley~\cite{Findley08} to $\xi$ such that $\xi$ is a Lebesgue point of $\mu$ and $f_\mu$ obeys a local Szeg\H o condition in a neighborhood of $\xi$.  Our Theorem~\ref{thmImaginaryLimitJacobi} shows that regularity or the local Szeg\H o condition are not needed to conclude bulk universality at scale $K_n(\xi,\xi)$. 

Theorem~\ref{thmImaginaryLimitJacobi} generalizes all previous pointwise results on bulk universality \eqref{genuniv} at scale $\tau_n = K_n(\xi,\xi)$.  Bulk universality limits are sometimes formulated with a factor of $n$ inside the kernel and a factor of $K_n(\xi,\xi)$ outside. In most cases this is merely a presentational choice, since the results are only proved to hold within a scope where \eqref{linearscale1} holds; there are also examples  \cite{BSing,Danka17} for which \eqref{linearscale1} isn't proved or expected; such mixed scale bulk universality limits don't fall within the scope of this paper. 

One of the standard applications of universality is to the fine structure of zeros near a point $\xi \in \supp \mu$. We denote by $\xi_j^{(n)}(\xi)$ for $j\in \bbZ$ the zeros of $p_n$ counted from $\xi$, i.e.,
\[
\dots < \xi_{-2}^{(n)}(\xi) <  \xi_{-1}^{(n)}(\xi) < \xi \le  \xi_{0}^{(n)}(\xi) <  \xi_{1}^{(n)}(\xi) <  \dots
\]
with no zeros of $p_n$ between $\xi_j^{(n)}$ and $\xi_{j+1}^{(n)}$. By the Freud--Levin theorem \cite{Freud,LevinLubinsky08} (see also \cite[Theorem~1.2]{AvilaLastSimon}),  the bulk universality limit \eqref{bulkuniversality3aug} implies
\begin{equation}\label{3aug2}
\lim_{n \to \infty} f_\mu(\xi) K_n(\xi,\xi) ( \xi_{j+1}^{(n)}(\xi) - \xi_{j}^{(n)}(\xi) ) = 1\qquad \forall j \in \bbZ.
\end{equation}
(when it occurs at scale $n$, this uniform spacing of zeros is described as clock behavior). Avila--Last--Simon \cite{AvilaLastSimon} proved that \eqref{bulkuniversality3aug} and \eqref{3aug2} hold a.e.\ on $\Sigma_\ac(\mu)$ for measures with ergodic Jacobi parameters; we prove that these phenomena hold a.e.\ on $\Sigma_\ac(\mu)$ for arbitrary measures $\mu$, thereby solving a conjecture of Avila--Last--Simon \cite[p. 83]{AvilaLastSimon}:

\begin{corollary}
For any measure $\mu$ obeying Hypothesis~\ref{hypothesismu},
 \eqref{bulkuniversality3aug} and \eqref{3aug2} hold for Lebesgue-a.e.\ $\xi \in \Sigma_\ac(\mu)$. 
\end{corollary}

Theorem~\ref{thmImaginaryLimitJacobi} can also be formulated uniformly for a set of energies $J \subset \bbR$:
\begin{theorem}\label{thmImaginaryLimitJacobi2}
Let $\mu$ obey Hypothesis~\ref{hypothesismu} and fix a compact set $J \subset\bbR$. Assume that for some $0 < \alpha < \pi/2$, \eqref{limitImm} holds uniformly in $\xi \in J$ with $0 < f_\mu(\xi) < \infty$ for $\xi \in J$. Then uniformly on compact regions of $(\xi,z,w) \in J \times \bbC \times \bbC$, \eqref{bulkuniversality3aug} holds.
\end{theorem}

The results presented so far are merely one facet of our approach; that approach naturally takes place in the setting of a matrix version of the Christoffel--Darboux kernel and links universality limits with normal limits of the $m$-function.  To formulate this, we recall that $\mu$ corresponds to a sequence of Jacobi parameters $a_n > 0$, $b_n \in\bbR$, $n =1,2,3, \dots$ such that
\[
z  p_n(z) = a_{n} p_{n-1}(z) + b_{n+1} p_n(z) + a_{n+1} p_{n+1}(z)
\]
with the convention $p_{-1}(z) = 0$. This second order recursion can be written as a first-order system, and iterating that system gives transfer matrices in $\SL(2,\bbC)$, 
\begin{equation}\label{3jul2}
B(n,z) = A(a_n,b_n;z) \cdots A(a_1, b_1;z), \quad A(a,b;z) = \begin{pmatrix}
\frac{z-b}{a}  &  - \frac 1{a} \\
a  & 0
\end{pmatrix}
\end{equation}
whose entries are
\[
B(n,z) = \begin{pmatrix}
p_n(z) & -q_n(z) \\
a_n p_{n-1}(z) & -a_n q_{n-1}(z)
\end{pmatrix}
\]
where $q_n$ denote the second kind polynomials for $\mu$, which are defined by
\[
q_n(z) = \int \frac{p_n(z) - p_n(\xi) }{ z - \xi} d\mu(\xi)
\]
for $n=0,1,2,\dots$ and $q_{-1}(z) = -1$. 
This is the standard transfer matrix formalism for Jacobi recursion \cite[Section 3.2]{SimonSzegoBook}; the same transfer matrices are central to the Riemann--Hilbert approach to orthogonal polynomials \cite[Section 3.2]{DeiftOPandRM}. We define
\begin{equation}\label{eq:CDKernelJForm0}
\cK_n(z, w) =  \begin{pmatrix}
\sum_{j=0}^{n-1} p_j(z) \overline{p_j(w)}  & \sum_{j=0}^{n-1}  q_j(z) \overline{p_j(w)}  \\
\sum_{j=0}^{n-1}  p_j(z) \overline{q_j(w)}  & \sum_{j=0}^{n-1} q_j(z) \overline{q_j(w)} 
\end{pmatrix}
\end{equation}
and call this the matrix Christoffel--Darboux (CD) kernel. Its upper left entry is precisely the CD kernel; moreover, through the other entries, the behavior of the matrix CD kernel also describes the behavior of second kind polynomials and arbitrary eigensolutions of the Jacobi matrix.

Since $m$ is a Herglotz function (analytic map $\bbC_+ \to \bbC_+$), its normal limits can only take values in the closure
\[
\overline{\bbC_+} = \bbC_+ \cup \bbR \cup \{\infty\}.
\]
It is common to study bulk universality in the presence of a.c. spectrum, corresponding to limits $\eta \in \bbC_+$. However, in our results below, the case $\eta \in \bbR \cup \{\infty\}$ is also allowed, and it corresponds to qualitatively different behavior.

Likewise, trace-normalized real symmetric positive $2\times 2$ matrices can be parametrized by a parameter $\eta \in \overline{\bbC_+}$. For $\eta \in \overline{\bbC_+}$, we denote
	\begin{equation}\label{4jul1}
\begin{split}
	\mathring H_\eta  & := \frac 1{1+\lvert \eta \rvert^2} \begin{pmatrix}
	1 & - \Re \eta \\ -\Re\eta & \lvert \eta \rvert^2
	\end{pmatrix} \quad  \eta \in \bbC_+ \cup \bbR \\
	\mathring H_\infty & := \begin{pmatrix}
	0&0\\0&1
	\end{pmatrix}.
\end{split}
\end{equation}
Denote also
\begin{equation}\label{definitionj2aug}
j =  \begin{pmatrix}
	0&-1\\1&0
	\end{pmatrix}
\end{equation}
and define
\begin{align*}
\mathring M_\eta(z)=e^{zj\mathring H_\eta}
\end{align*}
and
\[
\mathring \cK_\eta(z,w) = \int_0^1 e^{-t \overline{w} \mathring H_\eta j} \mathring H_\eta e^{t z j \mathring H_\eta}\,dt.
\]
Note that this kernel can be computed much more explicitly. If we denote
\[
h_\eta := \frac{\Im \eta}{1+|\eta|^2}\geq 0\text{ for }\eta \in\bbC_+, \quad h_\eta = 0\text{ for }\eta \in \bbR \cup\{\infty\}
\]
then we have
\[
\mathring \cK_\eta(z,w)= \frac{j(\cos(h_\eta(\overline{w}-z))-1)+\frac{ \mathring H_\eta}{h_\eta}\sin(h_\eta(\overline{w}-z))}{\overline{w}-z}, \qquad \text{for }\eta \in \bbC_+,
\]
\[
\mathring \cK_\eta(z,w) = \mathring H_\eta, \qquad \text{for }\eta \in \bbR \cup \{\infty\}.
\]
The quantities $\mathring M_\eta, \mathring \cK_\eta$ will appear in our results as universality limits; they emerge in the proofs as reproducing kernels of canonical systems (see Section 2) with constant trace-normalized Hamiltonians equal to $\mathring H_\eta$, evaluated at unit length.

In particular, in some approaches to universality the sinc kernel seems to appear through its Fourier transform; in our approach it appears as an eigensolution of a constant coefficient canonical system, normalized by an initial condition. Note that  if $\eta=i$, that is $ \mathring H_\eta = \frac{1}{2}I_2$, we obtain the well known Paley-Wiener (namely sinc) kernel. Moreover, note that the $\cos$ part only appears on the off diagonal and thus does not appear in the scalar reproducing kernels.

We show that normal limits of the $m$-function at $\xi \in \bbR$ imply a universality limit for the matrix Christoffel--Darboux kernel $\cK_n$, with scale
\[
\tau_\xi(n) = \tr \cK_n(\xi,\xi).
\]
From now on, we will formulate all results in a version uniform over a compact set $J \subset \bbR$; a singleton set $J = \{\xi\}$ is of course allowed.

\begin{theorem} \label{thmlimitmatrixCDkernel2}
Let $\mu$ obey Hypothesis~\ref{hypothesismu} and fix a compact set $J  \subset \bbR$. If $m$ has a normal limit uniformly on $J$, i.e., there exists $\eta : J \to \overline{\bbC_+}$ such that
\begin{equation}\label{uniformnormallimit}
\lim_{y \downarrow 0}   m(\xi + i y) = \eta(\xi)
\end{equation}
uniformly in $\xi \in J$, then
\begin{equation}\label{bulkuniversality3augmatrix}
\lim_{n\to\infty} \frac{1}{\tau_\xi(n)}\cK_n\left( \xi + \frac z{\tau_\xi(n)}, \xi + \frac w{\tau_\xi(n)} \right) =  \mathring \cK_{\eta(\xi)}(z,w),
\end{equation}
uniformly on  compact regions of $(\xi,z,w) \in J \times \bbC \times \bbC$.
\end{theorem}

Since we allow $\eta =\infty$,  \eqref{uniformnormallimit} should be understood with respect to the chordal metric on the Riemann sphere $\hat\bbC$.

\begin{remark}
\begin{enumerate}[(a)]
\item Evaluating \eqref{bulkuniversality3augmatrix} at $z=w=0$ shows
\begin{equation}\label{xixi}
\lim_{n\to\infty} \frac 1{\tau_\xi(n)} \cK_n(\xi,\xi) = \mathring \cK_{\eta(\xi)} (0,0) = \mathring H_{\eta(\xi)}.
\end{equation}
In particular, passing to the $(1,1)$-entry gives
\begin{equation}\label{eqn:scalechange}
\lim_{n\to\infty} \frac{K_n(\xi,\xi)}{\tau_\xi(n)} = \frac{1}{1+ \lvert \eta(\xi) \rvert^2}.
\end{equation}
Thus,   when $\eta(\xi) \neq \infty$, it is trivial to pass from the scaling by $\tau_\xi(n)$ to the scaling by $K_n(\xi,\xi)$. 
\item \eqref{eqn:scalechange} recovers a central result of subordinacy theory \cite{GilbertPearson,JitomirskayaLast}: written in terms of polynomials, it means
\[
\frac{ \sum_{j=0}^{n-1} p_j(\xi)^2 }{ \sum_{j=0}^{n-1} (p_j(\xi)^2 + q_j(\xi)^2 )} \to \frac{1}{1+ \lvert \eta(\xi) \rvert^2},
\]
so the eigensolution $p_n(\xi)$ is subordinate if  $\eta(\xi) = \infty$  and not subordinate if $\eta(\xi)  \in \bbC_+ \cup \bbR$. Similar considerations starting from \eqref{xixi}  show that a subordinate solution exists at energy $\xi$ if and only if $\eta(\xi) \in \bbR \cup \{\infty\}$.
\end{enumerate}
\end{remark}

Theorem~\ref{thmlimitmatrixCDkernel2} can even be turned into an equivalence statement, if the kernels $\cK_n$ are embedded in a continuous family of kernels by a linear interpolation. We define for $L \ge 0$
\begin{equation}\label{linearinterpolationkernel}
\cK_L(z,w) = \cK_{\lfloor L \rfloor} (z,w) + (L - \lfloor L \rfloor) ( \cK_{\lfloor L \rfloor+1} (z,w) - \cK_{\lfloor L \rfloor} (z,w))
\end{equation}
The resulting kernels interpolate between projections to subspaces $\spann\{1,\dots,z^{n-1}\}$. 
We also point out for $\xi \in \bbR$ the objects
\[
M_L(z,\xi) = I + (z-\xi) j \cK_L(z,\xi),
\]
because evaluated at integers, they give
\begin{equation}\label{variationofparametersfromkernel}
j_1 M_n(z,\xi) j_1 = B(n, \xi)^{-1} B(n,z), \qquad j_1 = \begin{pmatrix} -1 & 0 \\ 0 & 1 \end{pmatrix}.
\end{equation}
Thus, the matrices $M_L$ serve as a linear interpolation for the matrices $M_n$, which were studied before in studies of universality. In spectral theory the linear interpolation \eqref{linearinterpolationkernel} has been used for $\cK_L(\xi,\xi)$ in subordinacy theory for Jacobi matrices \cite{GilbertPearson,JitomirskayaLast}.  

For the continuous family of kernels $\cK_L$, we prove that universality at the scale
\begin{equation}\label{universalityscale}
\tau_\xi(L) = \tr \cK_L(\xi,\xi)
\end{equation}
is equivalent to a local condition: the existence of a normal limit of $m$ at $\xi$. We also point out some other equivalent criteria.

\begin{theorem} \label{theoremJacobiEquivalenceUniform}
Let $\mu$ obey Hypothesis~\ref{hypothesismu} and fix a compact $J \subset \bbR$.
The following are equivalent:
	\begin{enumerate}[(i)]
		\item  
		$m$ has a normal limit uniformly on $J$, i.e.,  \eqref{uniformnormallimit} holds for some $\eta: J \to \overline{\bbC_+}$.
		\item  
		For some  $H: J \to \Mat(2,\bbC)$, uniformly in $\xi \in J$,
		\[
		\lim_{n\to\infty} \frac{1}{\tau_\xi(n)} \cK_n(\xi,\xi) =  H(\xi).  
		\]
		\item  
		For some $M: J \times \bbC \to \SL(2,\bbC)$, uniformly on compact subsets of $(\xi,z)\in J \times \bbC$,
		\[
		\lim_{L\to\infty} M_L(\xi + z / \tau_\xi(L) , \xi) = M(\xi,z). 
		\]
		\item 
		For some $\cK:J \times \bbC \times \bbC \to \Mat(2,\bbC)$, uniformly on compact subsets of $(\xi,z,w) \in J \times \bbC \times \bbC$,
		\[
		\lim_{L \to\infty} \frac{1}{\tau_\xi(L)}\cK_L(\xi + z/\tau_\xi(L), \xi + w/\tau_\xi(L)) = \cK(\xi,z,w).		\] 
	\end{enumerate}
	Moreover, in this case, $H(\xi) = \mathring H_{\eta(\xi)}$, $M(\xi,z) = \mathring M_{\eta(\xi)}(z)$, and $\cK(\xi,z,w) = \mathring \cK_{\eta(\xi)}(z,w)$.
\end{theorem}

The equivalence of (ii) and (iv) in Theorem~\ref{theoremJacobiEquivalenceUniform} can be compared with a second approach of Lubinsky \cite{LubinskyJourdAnalyse08,LubinskyJFA09}. Lubinsky proves that for a measure $\mu$ with $\xi \notin \supp\mu_s$ and $C^{-1} \le f_\mu \le C$ in a neighborhood of $\xi$, a bulk universality limit holds for all $z,w$ if and only if it holds on the diagonal $z=w$. Our equivalence of (ii) and (iv) shows \emph{with no prior assumptions} that the matrix bulk universality limit holds for all $z,w$ if and only if it holds at the center $z=w=0$.

Breuer, Last and Simon \cite{BreuerLastSimon} showed that under the assumption that $\xi$ is a Lebesgue point and $m$ has a finite normal limit at $\xi$, bulk universality at scale $n$ of the CD kernel implies bulk universality at scale $n$ of the matrix CD kernel. Recently, Breuer \cite{Breuer2021ConstrApprox} showed that under the same assumptions, bulk universality with scale $n$ is equivalent to convergence of $B(n,\xi)^{-1} B(n,\xi+ z/n)$. These results follow effortlessly from Theorem~\ref{theoremJacobiEquivalenceUniform};  namely, a finite normal limit for $m$ already gives a bulk universality limit at scale $\tau_\xi(n)$, and it implies by \eqref{eqn:scalechange} that the scales $K_n(\xi,\xi)$ and $\tau_\xi(n)$ are proportional. Thus, one of those scales is proportional to $n$ if and only if the other is.

In fact, our results naturally take place in a more general setting, with transfer matrices considered as the basic object. This will explain why our approach does not need compact support of the measure. Moreover, this approach is to some extent model-independent and applies to other self-adjoint settings with $2\times 2$ transfer matrix formalisms such as Schr\"odinger and Dirac operators; we note however that it does not automatically apply to orthogonal polynomials on the unit circle, due to their unitary nature.

To state this general setting, we redefine all the basic objects; we will do so while reusing notation, and in Section 2 we will explain the compatibility of the definitions above and below.

\begin{hypothesis}\label{Txzcanonicalsystem}
Consider matrix-valued functions $A, B : [0,\infty) \to \Mat(2,\bbC)$ which are locally integrable in the sense that their entries are in $L^1([0,x])$ for all $x  < \infty$, and have the property that $A(x) \ge 0$, $B(x)^* = B(x)$, and $\tr(A(x) j) = \tr(B(x) j)= 0$ for Lebesgue-a.e.\ $x$, with $j$ defined by \eqref{definitionj2aug}. Let  $T: [0,\infty) \times \bbC \to \Mat(2,\bbC)$ be the solution of the initial value problem
\begin{equation} \label{canonicalsystemAB}
j\partial_xT(x,z)=(-z A(x)+B(x))T(x,z),\quad T(0,z)=I_2
\end{equation}
and  assume that (limit point case)
\begin{equation}\label{limitpointcondition}
\tr \int_0^\infty T(x,0)^* A(x) T(x,0) dx = \infty.
\end{equation}
\end{hypothesis}
The condition $\tr(A(x) j) = \tr(B(x) j)= 0$ could equivalently be stated as $A, B : [0,\infty) \to \Mat(2,\bbR)$. We chose the former formulation, since it is uniform for different signature matrices $j$.

Initial value problems of the form \eqref{canonicalsystemAB} are called canonical systems or Hamiltonian systems \cite{BLY1,KrallBook}.  The solutions $T(x,z)$ are jointly continuous,  entire in $z$ for each $x$, and $\det T(x,z) = 1$ for all $x, z$.  We define the corresponding kernels  by
\begin{equation}\label{kernelfromTdefn}  
\cK_L(z,w) = \int_0^L T(x,w)^* A(x) T(x,z) \,dx.
\end{equation}
Moreover, the solutions $T(x,z)$ have the $j$-monotonic property in the upper half-plane
\begin{equation}\label{eq:jMonotonic}
\frac{T(x_2,z)^*jT(x_2,z)-T(x_1,z)^*jT(x_1,z)}{i}\leq 0, \qquad z \in \bbC_+, \quad 0 \le x_1 \le x_2
\end{equation}
and the $j$-unitary property on the real line,
\begin{equation}\label{eq:jUnitary}
\frac{T(x,z)^*jT(x,z)-j}{i}= 0, \qquad z \in \bbR, \quad x \ge 0.
\end{equation}
The condition \eqref{eq:jUnitary} is equivalent to  $T(x,z) \in \SL(2,\bbR)$ for $z\in \bbR$.

Due to \eqref{eq:jMonotonic}, Weyl disks can be introduced in this setting. By a standard abuse of notation, we will use the same notation for an $\SL(2,\bbC)$ matrix and for the M\"obius transformation it generates on the Riemann sphere $\hat\bbC$, with the standard projective identification of $w\in\bbC$ with the coset of $\binom w1$ and $\infty$ with the coset of $\binom 10$. We denote $\overline{\bbC_+} = \bbC_+ \cup \bbR \cup \{\infty\}$. With this notation in mind, for any $z\in \bbC_+$, the Weyl disks are defined by
\[
D(x,z) = \{ w \in \hat\bbC \mid T(x,z) w \in \overline{\bbC_+} \}
\] 
Due to \eqref{eq:jMonotonic}, the Weyl disks are nested, $D(x_2, z) \subset D(x_1, z)$ for $x_1 \le x_2$. Thus, for each $z\in \bbC_+$, the intersection $\cap_x D(x,z)$ is a disk or a point. It is a point if and only if \eqref{limitpointcondition} holds, and in this case, the Weyl disks define an analytic map $m: \bbC_+ \to \overline{\bbC_+}$ by
\begin{equation}\label{WeylIntersection}
\{ m(z) \} = \bigcap_{x \ge 0} D(x,z).
\end{equation}

\begin{theorem} \label{thmuniformnormallimit}
Let $T(x,z)$ be as in Hypothesis~\ref{Txzcanonicalsystem}. Let $m$ be the corresponding $m$-function \eqref{WeylIntersection}, $\cK_L$ the corresponding kernels \eqref{kernelfromTdefn}, and $\tau_\xi(L)$ the scaling factors \eqref{universalityscale}. The following are equivalent:
	\begin{enumerate}[(i)]
		\item  
		$m$ has a normal limit uniformly on $J$, i.e.,  \eqref{uniformnormallimit} holds for some $\eta: J \to \overline{\bbC_+}$.
		\item 
		For some  $H: J \to \Mat(2,\bbC)$, uniformly in $\xi \in J$,
		\[
\lim_{L\to\infty}		\frac{1}{\tau_\xi(L)} \cK_L(\xi,\xi)  =  H(\xi). 
		\]
		\item 
		For some $M: J \times \bbC \to \SL(2,\bbC)$, uniformly on compact subsets of $(\xi,z)\in J \times \bbC$,
		\[
		\lim_{L\to\infty}  T(L,\xi)^{-1} T(L,\xi + z / \tau_\xi(L)). 
		= M(\xi,z) 
		\]
		\item 
		For some $\cK:J \times \bbC \times \bbC \to \Mat(2,\bbC)$, uniformly on compact subsets of $(\xi,z,w) \in J \times \bbC \times \bbC$,
		\[
		\lim_{L\to\infty} \frac{1}{\tau_\xi(L)}\cK_L(\xi + z/\tau_\xi(L), \xi + w/\tau_\xi(L)) = \cK(\xi,z,w).
		\] 
	\end{enumerate}
	Moreover, in this case, $H(\xi) = \mathring H_{\eta(\xi)}$, $M(\xi,z) = \mathring M_{\eta(\xi)}(z)$, and $\cK(\xi,z,w) = \mathring \cK_{\eta(\xi)}(z,w)$.
\end{theorem}

In this setting it also makes sense to extract the scalar kernel
\begin{equation}\label{frommatrixtoscalarkernel}
K_L(z,w) = \begin{pmatrix} 1 \\ 0 \end{pmatrix}^* \cK_L(z,w)  \begin{pmatrix} 1 \\ 0 \end{pmatrix}
\end{equation}
and explicitly state the corresponding version of Theorem~\ref{thmImaginaryLimitJacobi}:

\begin{theorem}\label{thmImaginaryLimitCanonical}
Let $T(x,z)$ be as in Hypothesis~\ref{Txzcanonicalsystem}. Let $m$ be the corresponding $m$-function \eqref{WeylIntersection} and $K_L$ the corresponding kernels given by \eqref{kernelfromTdefn}, \eqref{frommatrixtoscalarkernel}. Fix a compact $J \subset \bbR$ and assume that for some $0 < \alpha < \pi/2$, \eqref{limitImm} holds uniformly in $\xi \in J$. Then uniformly on compact regions of $(\xi, z,w) \in J \times \bbC \times \bbC$, 
\begin{equation}\label{bulkuniversality3augL}
\lim_{L\to\infty}\frac{K_L\left(\xi+\frac{z}{f_\mu(\xi) K_L(\xi,\xi)},\xi+\frac{w}{f_\mu(\xi) K_L(\xi,\xi)}\right)}{K_L(\xi,\xi)}=\frac{\sin(\pi(\overline{w}-z))}{\pi(\overline{w}-z)}.
\end{equation}
\end{theorem}

These results have immediate applications to the Schr\"odinger equation
\begin{equation}\label{Schrodinger}
-\partial_x^2y(x,z)+V(x)y(x,z)=zy(x,z),
\end{equation}
with real-valued function $V$ such that $V\in L^1([0,x])$ for all $x < \infty$; universality limits for Schr\"odinger operators were previously considered in a special case in \cite{MaltsevCMP}.  To account for a boundary condition 
\begin{equation}\label{boundarycondition2aug}
\cos \beta f(0) + \sin\beta f'(0) = 0 
\end{equation}
at the regular endpoint at $0$, we fix $\beta \in [0,\pi)$ and consider the eigensolutions $\phi(x,z)$, $\theta(x,z)$ at energy $z$, with initial conditions
\[
\begin{pmatrix}
\partial_x \phi(0,z)& \partial_x \theta(0,z)\\
\phi(0,z)& \theta(0,z)
\end{pmatrix} = R_\beta^{-1}, \qquad R_\beta = \begin{pmatrix}
\cos \beta & - \sin \beta \\
\sin \beta & \cos \beta
\end{pmatrix}
\]
(for $\beta= 0$, these are the Dirichlet and Neumann solutions, respectively). These give rise to the transfer matrices
\[
T(x,z)= R_\beta \begin{pmatrix}
\partial_x \phi(x,z)& \partial_x \theta(x,z)\\
\phi(x,z)& \theta(x,z)
\end{pmatrix}
\]
which obey an equation of the form \eqref{canonicalsystemAB} (see Section 2). As we will explain, a calculation gives   
\begin{equation}\label{SchrodingerMatrixCDKernel}
\cK_L(z,w)=\begin{pmatrix}
\int_0^L \phi(x,z) \overline{\phi(x,w)} \,dx & \int_0^L \theta(x,z) \overline{\phi(x,w)} \,dx  \\
\int_0^L \phi(x,z) \overline{\theta(x,w)} \,dx & \int_0^L \theta(x,z) \overline{\theta(x,w)} \,dx  
\end{pmatrix}.
\end{equation}
Its $(1,1)$-entry
\begin{equation}\label{CDkernelSchrodinger}
K_L(z,w) = \int_0^L \phi(x,z)  \overline{\phi(x,w)} dx
\end{equation}
gives reproducing kernels for subspaces in the spectral representation of the Schr\"odinger operator. In particular, Theorems~\ref{thmuniformnormallimit}, \ref{thmImaginaryLimitCanonical} immediately apply to any half-line Schr\"odinger operator in the limit point case at $\infty$. For concreteness, we state the corresponding version of Theorem~\ref{thmImaginaryLimitCanonical}:

\begin{theorem}\label{thmImaginaryLimitSchrodinger}
Consider a real-valued potential $V$ with $V \in L^1([0,x])$ for all $x < \infty$ which is limit point at $\infty$. Fix $\beta \in [0,\pi)$ and consider the operator $H_{V,\beta} = - \frac{d^2}{dx^2} + V$ on $L^2([0,\infty))$ with domain
\begin{align*}
D(H_{V,\beta}) = \{ f \in L^2([0,\infty)) \mid & f, f' \in \AC([0,x])\,\forall x < \infty,   -f'' + Vf \in L^2([0,\infty)) , \eqref{boundarycondition2aug}\text{ holds} \}.
\end{align*}
Let $m(z)$ be its Weyl $m$-function. Let $J \subset \bbR$ and assume  that for some $0 < \alpha < \pi/2$, the limit \eqref{limitImm} converges uniformly in $\xi \in J$ and $0 < f_\mu(\xi) < \infty$. Then uniformly on compact regions of $(\xi,z,w) \in J \times \bbC \times \bbC$, the kernel \eqref{CDkernelSchrodinger} obeys \eqref{bulkuniversality3augL}.
\end{theorem}

For a probability measure $\mu$ on the unit circle $\partial\bbD$ whose support is not a finite set, orthonormal polynomials on the unit circle (OPUC) are denoted by $\varphi_n(z)$ and once again defined by the Gram--Schmidt process applied to the sequence $\{z^n \}_{n=0}^\infty$.  The reflected polynomials are
\[
\varphi_n^*(z) = z^n \overline{ \varphi_n (1 / \overline{z} ) }
\]
and the Christoffel--Darboux kernel for OPUC is given by the formula
\begin{equation}\label{CDkernelOPUC1}
k_n(z,w) = \frac{ \varphi_n^*(z) \overline{ \varphi_n^* (w) } -  \varphi_n(z) \overline{ \varphi_n (w) } }{ 1 - z \overline{w} },
\end{equation}
see, e.g, \cite{SimonCDkernel}. As before, there are different conventions resulting in $n+1$ instead of $n$ in this formula, and different conventions for placement of the complex conjugate. 

The Carath\'eodory function corresponding to $\mu$ is defined by
\[
F(z) = \int_{\partial\bbD} \frac{e^{i\theta} + z}{e^{i\theta} - z} \, d\mu(e^{i\theta}), \qquad z \in \bbD.
\]
We prove the following:

\begin{theorem}\label{thmOPUC}
Fix a compact set $J \subset\bbR$. Assume that for some $0 < \alpha < \pi/2$,
\begin{equation}\label{nontangentiallimitOPUC}
g_\mu(\xi) := \lim\limits_{\substack{z \to e^{i\xi} \\ -\alpha \le \arg (1 - z e^{-i\xi}) \le  \alpha}} \Re F(z)
\end{equation}
holds uniformly in $\xi \in J$ with $0 < g_\mu(\xi) < \infty$ for $\xi \in J$. Then uniformly on compact regions of $(\xi,z,w) \in J \times \bbC \times \bbC$,
\begin{equation}\label{OPUCconclusion}
\lim_{n\to\infty}\frac{e^{-in \frac{z - \overline{w}}{ 2 g_\mu(\xi) k_n(e^{i\xi},e^{i\xi})} } k_n\left(e^{i(\xi+\frac{z}{g_\mu(\xi) k_n(e^{i\xi},e^{i\xi})})},e^{i(\xi+\frac{w}{g_\mu(\xi) k_n(e^{i\xi},e^{i\xi})})} \right)}{k_n(e^{i\xi},e^{i\xi})}=\frac{\sin( \frac 12 (\overline{w}-z))}{\frac 12(\overline{w}-z)}
\end{equation}
\end{theorem}

The exponential prefactor can be traced back to an asymmetry in the setup: the kernel \eqref{CDkernelOPUC1} is the reproducing kernel for $\spann \{ z^j \mid j = 0,1,\dots, n-1\}$. If the result was written in terms of reproducing kernels for $\spann \{ z^{j - \lfloor n\rfloor/2} \mid j = 0,1,\dots, n-1\}$, which would be related to the CMV basis, this factor would not be needed. Of course, this prefactor can be simplified in cases when $k_n(e^{i\xi},e^{i\xi}) / n$ has a limit, as is the case in prior results \cite{LevLubOPUC,Findley08,LubinskyNguyen,Simanek18}. The linear behavior of the Christoffel function is studied in great generality in \cite{Totik14}. For an OPUC analog of the result of \cite{MorenoFinkelSousaConstr}, see \cite{Simanek17,BNR}.

As on the real line, it is well known \cite{RudinRealandComplex} that \eqref{nontangentiallimitOPUC} is the density of $\mu$ with respect to Lebesgue measure on $\bbR$, and that its pointwise existence is implied by Lebesgue point conditions.

Let us now outline our approach. Matrices  $M_L(z,\xi) = T(L,\xi)^{-1} T(L,z)$ have been used before in considerations of universality. They can be motivated by variation of parameters; more substantially, since $T(L,\xi) \in \SL(2,\bbR)$ and $M_L(\xi,\xi) = I_2$, they can be viewed as a gauge transformation of a $j$-inner entire function $T(L,z)$ into the Potapov--de Branges gauge. If for the purpose of this motivation we assume that $\xi = 0$, then this is exactly in the setting of de Branges canonical systems \cite{deBrangesHilbertSpace,RemlingBookCanSys}, which are initial value problems of the form
\begin{equation} \label{eq:genCanonical}
	j\partial_xM(x,z)=-zH(x)M(x,z)\quad M(0,z)=I_2,
\end{equation}
with entries of $H$ locally integrable, $H(x) \ge 0$ and $\tr (Hj) = 0$. We review elements of this theory in Section~\ref{sec:CanSys}.

Canonical systems are very natural objects from the perspective of inverse spectral theory. The main reason for this is the de Branges correspondence, which can be seen as a (more sophisticated) analog of Favard's theorem for canonical systems. That is, up to some normalizations discussed in Section \ref{sec:CanSys}, for any Herglotz function there exists a unique associated canonical system. In his foundational paper \cite{LubinskyAnnals}, Lubinsky developed the idea to conclude from convergence on the diagonal, convergence off the diagonal by comparing it with a reference measure and using an inequality between reproducing kernels called Lubinky's inequality. In this approach, it is necessary to construct a new measure which dominates both, the original measure and the reference measure. Due to Favard's theorem this is a powerful tool for orthogonal polynomials and has been used  by several authors \cite{SimonTwoExt08,TotikUniv09}. However, for classes of differential operators without Favard's theorem, this approach causes difficulties, whereas our theorems automatically apply. In Section~\ref{sectionContinuity} we describe the homeomorphisms which are derived from the de Branges correspondence.

Besides the de Branges correspondence, our paper relies on the flexibility of canonical systems to perform operations which would not be possible within the setting of orthogonal polynomials.

The proof of Theorem~\ref{thmuniformnormallimit} relies on a rescaling trick at the level of the canonical systems. Namely, the action of rescaling the spectral parameter $z$ in the $m$ function is explicitly linked to a rescaling of the $x$-variable for a trace-normalized canonical system. The scaling trick was initally found by Kasahara \cite{Kasahara76} for Krein strings and used by Eckhardt--Kostenko--Teschl \cite{EckKostTeschl} and Langer--Pruckner--Woracek \cite{LangerPruckWor} for canonical systems to investigate large energy asymptotics of the $m$-function. Our realization is that this can also be used in the other direction to prove matrix universality, by ``zooming in'' towards $\xi \in \bbR$ instead of out towards $\infty$. Also, whereas they use the rescaling trick to conclude properties of the $m$-function, we use the $m$-function to conclude universality limits.  This will be presented in Section~\ref{sectionLimitm}.

We also emphasize that although the matrix CD kernel leads to particularly elegant results, the results for the scalar CD kernel are substantially stronger, because the imaginary part of a Herglotz function has substantially better boundary behavior than the real part; for instance, even for purely a.c. measures with continuous density, $\Re m$ doesn't necessarily have finite normal limits. Thus, removing the influence of $\Re m$ was of great interest to us. In order to achieve this, we had to find a new shifted rescaling trick with a shift in the value of $m$, which is presented in Section~\ref{sectionLimitImm}.

The application to orthogonal polynomials on the unit circle requires some additional arguments; this is presented in Section~\ref{sectionOPUC}.

\subsection*{Acknowledgements}
We are grateful to Harald Woracek and Peter Yuditskii for very helpful discussions. 

\section{Canonical systems and Christoffel--Darboux kernels}\label{sec:CanSys}

\subsection{Christoffel--Darboux kernels and Weyl disks}

Let us start in the setting of Hypothesis~\ref{Txzcanonicalsystem}. The solution $T(x,z)$ is considered in the sense of the integral equation
\[
T(L,z) = I_2 - \int_0^L j (-z A(x) + B(x)) T(x,z) \,dx.
\]
In particular, $T$ is locally absolutely continuous in $x$ for each $z$, jointly continuous in $x$ and $z$, and entire in $z$ for each $x$. Jacobi's formula for the derivative of the determinant shows $\det T(x,z) = 1$.

Differentiating the form $T(x,w)^* j T(x,z)$ with respect to $x$ yields
\begin{align}\label{eq:TransferJForm}
T(L,w)^*j T(L,z)-j=(\overline{w}-z)\int_0^L T(x,w)^* A(x) T(x,z)\, dx.
\end{align}
With kernels defined by  \eqref{kernelfromTdefn}, this gives an analog of the Christoffel--Darboux formula,
\begin{equation}\label{eq:CDKernelJForm}
\cK_L (z,w) = \frac{ T(L,w)^*j T(L,z)-j }{\overline{w} - z}.
\end{equation}
Moreover, the case $w = z  \in \bbR$ in \eqref{eq:TransferJForm} shows the $j$-unitary property \eqref{eq:jUnitary}, i.e., $T(L,z) \in \SL(2,\bbR)$ for $z\in \bbR$, and the case $w = z \in \bbC_+$ shows the $j$-monotonic property \eqref{eq:jMonotonic} which leads to the Weyl disk formalism.

The $j$-monotonic family $\{T(L,z)\}$ parametrized by $L \in [0,\infty)$ is subject to a \emph{gauge transformation} \cite{BLY1}
\[
\{T(L,z) \} \mapsto \{ U(L) T(L,z) \}
\]
where $U(L) \in \SL(2,\bbR)$ for each $L$. Recall that $U(L) \in \SL(2,\bbR)$ is equivalent to $U(L)^* j U(L) = j$, so gauge transformations don't affect $j$-monotonicity, the Weyl disks, $m$-function, or matrix CD kernels. In particular, choosing $U(L) = T(L,0)^{-1}$ leads to a new $j$-monotonic family
\begin{equation}\label{3jul3}
M(L,z) = T(L,0)^{-1} T(L,z)
\end{equation}
which is in the Potapov--de Branges gauge
\begin{equation}\label{eqn:PdB}
M(x,0) = I_2, \qquad \forall x.
\end{equation}
A direct computation then shows that $M(x,z)$ solves 
\begin{equation}\label{fromhamiltoniantocanonical}
j\partial_x M(x,z)=- z H(x)M(x,z),\quad H(x):=T(x,0)^*A(x)T(x,0).
\end{equation}
This reduces an arbitrary Hamiltonian system to a canonical system of the type \eqref{eq:genCanonical}.  It is also notable that the antiderivative of $H$ appears in the kernel: directly from
\eqref{kernelfromTdefn} it follows that
\[
\int_0^L H(x) \,dx = \int_0^L T(x,0)^* A(x) T(x,0) \,dx = \cK_L(0,0).
\]

Although we use $0$ as a special reference point in this section, later in our work, transformations of the form \eqref{3jul3} will be preceded by an affine transformation in the spectral parameter; instead of considering a family $T(L,z)$, we will study the families $T(L, \xi + z/r)$, where $\xi \in \bbR$ and $r > 0$ is a scaling parameter. Taking the limit as $r \to \infty$ will correspond to universality at the point $\xi$.

\subsection{Canonical systems and de Branges theorem}
The case $B = 0$ corresponds to the important special case of de Branges canonical systems; in this case we will follow the notation  \eqref{eq:genCanonical}. It is also common to refer to $H$ as a Hamiltonian. Throughout we only consider canonical systems on an interval $[0,L_0)$, where $L_0\in(0,\infty]$, with the property 
\begin{equation}\label{limitpointconditioncanonicalsystem}
\int_0^{L_0} \tr H(x)\, dx=\infty.
\end{equation}
This property precisely corresponds to the limit point case \cite{RemlingBookCanSys}, i.e., the case in which Weyl disks shrink to a point as $x \to L_0$, thereby defining the $m$-function or Weyl function of the canonical system by \eqref{WeylIntersection}. Comparing \eqref{limitpointconditioncanonicalsystem} with \eqref{fromhamiltoniantocanonical} explains why we could use \eqref{limitpointcondition} as a characterization of the limit point case in Hypothesis~\ref{Txzcanonicalsystem}.

It is very convenient to pass to trace-normalized canonical systems,  i.e., impose
\[
\tr H = 1
\]
Lebesgue-a.e.. To help notationally distinguish the settings, we will denote the parameter by $t$ or $T$ when discussing trace-normalized canonical systems. Due to the limit point assumption \eqref{limitpointconditioncanonicalsystem},  we always consider trace normalized Hamiltonians on the full half line $[0,\infty)$.

An inverse theory of canonical systems was studied by Potapov and de Branges \cite{PotapoovMulStruct,deBrangesHilbertSpace}.  In particular, if we restrict to trace-normalized canonical systems and consider as equal Hamiltonians which agree Lebesgue-a.e., there is a remarkable bijective correspondence due to de Branges:

\begin{theorem}[de Branges \cite{deBrangesHilbertSpace}; see also {\cite[Theorem 5.1]{RemlingBookCanSys}}] \label{deBrangesTheorem}
	Each Herglotz function is the $m$-function of a unique trace-normalized canonical system \eqref{eq:genCanonical}.
\end{theorem}

Any canonical system can be reduced to the trace-normalized case by a reparametrization of the interval, and this doesn't affect the $m$-function of the canonical system:

\begin{lemma}[{\cite[Proposition 1]{Romanov}}]\label{lem:Reparam} 
Let $\tilde H$ be the Hamiltonian of a limit-point canonical system on $[0,\tilde L_0)$ and $\tilde M$ the solution.
Then
\[
a(L) = \int_0^L \tr \tilde H(x) dx
\]
defines an increasing, locally absolutely continuous surjection $a: [0,\tilde L_0) \to [0,\infty)$, a function $M$ is uniquely defined by
\begin{equation}
\tilde M(x,z)= M(a(x),z),\label{eq:Reparam7}
\end{equation}
and $M(t,z)$ is the solution of a trace-normalized canonical system with the Hamiltonian $H$ obeying
\begin{align}
\tilde H(x)= H(a(x))a'(x). \label{eq:Reparam}
\end{align}
The two canonical systems have the same Weyl $m$-function.
\end{lemma}

\subsection{Jacobi recursion}
Let us consider the setting of Hypothesis~\ref{hypothesismu} and recall how it relates to canonical systems. Let $\{B(n,z)\}_{n=0}^\infty$ be the family of Jacobi transfer matrices \eqref{3jul2}. Then the matrices 
\[
T(n,z) = j_1   B(n,z) j_1, \qquad j_1 = \begin{pmatrix} -1 & 0 \\ 0 & 1 \end{pmatrix}
\]
form a $j$-monotonic family whose Weyl function is the same function $m(z)$ obtained as the Stieltjes transform \eqref{Stieltjes} of the orthogonality measure $\mu$ \cite[Section 3.2]{SimonSzegoBook}. By a direct calculation starting from \eqref{3jul2}, matrices $M(n,z)$ defined by \eqref{3jul3} obey
\[
M(n+1,z) = M(n,z) + z T(n,0)^{-1} N T(n,0)  M(n,z), \qquad N =  \begin{pmatrix} 0 & 0 \\ 1 & 0 \end{pmatrix}.
\]
This can be extended to a solution of a canonical system by linear interpolation: define for $x \ge 0$
\[
M(x,z) = M(\lfloor x \rfloor, z) + (x - \lfloor x \rfloor) ( M(\lfloor x \rfloor + 1, z)  - M(\lfloor x \rfloor, z)  ).
\]
This is the solution of a canonical system \eqref{eq:genCanonical} with piecewise constant Hamiltonian
\[
H(x) =  - j T({\lfloor x \rfloor},0)^{-1} N T({\lfloor x \rfloor},0) = e_{\lfloor x \rfloor}(0)^* e_{\lfloor x \rfloor}(0), \qquad e_n(z) = \begin{pmatrix} p_n(z) & q_n(z) \end{pmatrix}, 
\]
as is seen by a direct verification;  linear interpolation works because $T(n,0)^{-1} N T(n,0)$ is nilpotent.

Moreover, from \eqref{eq:genCanonical} and $H j H = 0$ we see that on intervals $[n,n+1]$,
\[
M(L,w)^* H(L) M(L,z) = M(n,w)^* H(n) M(n,z), \qquad L \in [n,n+1]
\]
and a direct calculation gives
\begin{equation}\label{singularintervalOP25jul}
M(n,w)^* H(n) M(n,z) = e_n(w)^* e_n(z)
\end{equation}
so if we follow the general definition of matrix CD kernels from \eqref{kernelfromTdefn}, \eqref{singularintervalOP25jul} shows that they are linearly interpolated,
\begin{equation}\label{linearinterpolationOPfromcanonical}
\cK_L(z,w) = \cK_n(z,w) + (L - n) e_n(w)^* e_n(z), \qquad L \in [n,n+1].
\end{equation}
Iterating this on intervals of length $1$ recovers the formula \eqref{eq:CDKernelJForm0}, and \eqref{linearinterpolationkernel} is immediate; thus, this explains that the general definition \eqref{kernelfromTdefn} of the matrix CD kernel  is compatible with the definition we gave for the orthogonal polynomial setting \eqref{eq:CDKernelJForm0}, \eqref{linearinterpolationkernel}.

\subsection{Schr\"odinger operators}
We briefly comment on the Schr\"odinger equation \eqref{Schrodinger}.  The second-order eigenfunction equation gives a first-order evolution of transfer matrices,
\[
j\partial_xT(x,z)=(-z A(x)+ B(x) )T(x,z),
\]
where a direct calculation using $R_\beta^{-1} = R_\beta^*$ gives
\[
A(x)= R_\beta \begin{pmatrix}
0&0\\0&1
\end{pmatrix}R_\beta^{*},\quad 
B(x)=R_\beta \begin{pmatrix}
-1& 0\\0& V(x)
\end{pmatrix} R_\beta^{*}.
\]
From here,
\[
T(x,w)^*  A(x)  T(x,z) = e_x(w)^* e_x(z), \qquad e_x(z) = \begin{pmatrix} \phi(x,z) & \theta(x,z) \end{pmatrix}
\]
which justifies the formula \eqref{SchrodingerMatrixCDKernel}.

\subsection{Constant Hamiltonians}
We are particularly interested in canonical systems with constant Hamiltonians. Constant trace-normalized Hamiltonians are parametrized by $\eta \in \overline{\bbC_+}$, i.e., they are of the form $H(t) = \mathring H_\eta$ for a unique $\eta \in \overline{\bbC_+}$.

Note that $\mathring H_\eta$ is singular if $\eta \in \bbR \cup \{\infty\}$. That is in this case we can write
\[
\mathring H_\eta =v_\eta v_\eta^*,\quad v_\eta =\frac{1}{\sqrt{1+ \eta^2}}\begin{pmatrix}
1\\
- \eta
\end{pmatrix}, \quad v_\infty = \begin{pmatrix} 0 \\ 1 \end{pmatrix}.
\]

\begin{example} \label{xmplconsthamiltonian}
Fix $\eta \in \overline{\bbC_+}$. The canonical system with the constant Hamiltonian $H(t) = \mathring H_\eta$, $t \in [0,\infty)$ has the following properties:
\begin{enumerate}[(a)]
\item the solution of the canonical system is
\begin{equation}\label{4jul2}
M(t,z) = e^{t jz \mathring H_\eta}
\end{equation}
This is also expressible as
\begin{equation}\label{4jul3}
M(t,z)= \begin{cases} \cos(t z h_\eta )I_2+\frac{\sin(tzh_\eta)}{h_\eta}j \mathring H_\eta & \eta \in \bbC_+ \\
I_2+tz j H_\eta & \eta \in \bbR \cup\{\infty\}
\end{cases}
\end{equation}
\item the $m$-function is a constant function, $m(z) = \eta$ for all $z\in \bbC_+$.
\item the kernel is
\[
\cK_t(z,w)=\frac{j(\cos(t h_\eta(\overline{w}-z))-1)+\frac{H_\eta}{h_\eta}\sin(t h_\eta(\overline{w}-z))}{\overline{w}-z}.
\]
\end{enumerate}
\end{example}

\begin{proof}
(a) \eqref{eq:genCanonical} has a unique solution, so the first claim follows by direct verification by plugging in \eqref{4jul2} into \eqref{eq:genCanonical}). Expanding \eqref{4jul2} into a power series, using
\begin{align}
(j\mathring H_\eta)^2= - h_\eta^2 I_2,
\end{align}
and separately summing even and odd terms, we obtain \eqref{4jul3}.

(b) Let us again use projective coordinates $\eta = \eta_1 / \eta_2$. By a direct calculation, $\binom{ \eta_1}{\eta_2}$ is an eigenvector of $j \mathring H_\eta$. Thus, it is an eigenvector of $M(t,z)$, so $M(t,z) \eta = \eta \in \overline{\bbC_+}$ in the projective sense. In other words $\eta$ is in each Weyl disk.  Since the canonical system is in the limit point case, this proves that $m(z) = \eta$.

(c) follows from \eqref{eq:CDKernelJForm} and \eqref{4jul3}, using $\mathring H_\eta j \mathring H_\eta =h_\eta^2j$.
\end{proof}

This explains the quantities $\mathring M_\eta$, $\mathring \cK_\eta$ from the introduction; they appear from canonical systems with constant trace-normalized Hamiltonians evaluated at $t = 1$.

\section{Continuity properties} \label{sectionContinuity}

The de Branges bijection, recalled in Theorem~\ref{deBrangesTheorem}, is central to our approach. To proceed, we have to define some metric spaces.

Let $X,Y$ be metric spaces; assume that $X$ is not compact but has compact subsets $\{S_n\}_{n=1}^\infty$ such that $S_n \subset \Int S_{n+1}$ and $\cup_{n\in\bbN} S_n = X$. Then the topology of locally uniform convergence on $C(X,Y)$ is metrizable, with one choice of metric given by
\[
d(f,g) =  \sum_{n=1}^\infty 2^{-n} \min \{ 1, \sup_{x \in S_n} d_Y( f(x), g(x)) \}.
\]

Let $\fM$ be the set of all analytic maps $\bbC_+ \to \overline{\bbC_+}$. Note that this is the set of Herglotz functions, enlarged by constants in $\bbR \cup \{\infty\}$. We view $\fM$ as a subset of $C(\bbC_+, \overline{\bbC_+} )$; this is the setting of Montel's theorem \cite[Chapter 6]{SimonBasicCompAna}, so $\fM$ is a compact metric space.

Let $\fH$ be the set of all trace-normalized Hamiltonians on $[0,\infty)$, with the topology of locally uniform convergence of their antiderivatives, i.e., the distance between Hamiltonians $H_1, H_2$ is given by
\[
d(H_1, H_2) =  \sum_{n=1}^\infty 2^{-n} \min \left\{ 1, \sup_{T \in [0,n]}  \left\lvert  \int_0^T H_1(t) dt -  \int_0^T H_2(t) dt \right\rvert \right\}.
\]

Let $\fS$ be the set of all solutions $M(t,z)$ of canonical systems with trace-normalized Hamiltonians, viewed as functions of $(t,z) \in [0,\infty) \times \bbC$, with the topology of locally uniform convergence, i.e., that inherited from $C( [0,\infty) \times \bbC, \Mat(2,\bbC) )$.

Let $\fK$ be the set of all kernels -- all functions $\cK_t(z,w)$, viewed as functions of $(t,z,w) \in [0,\infty) \times \bbC \times \bbC$, with the topology of locally uniform convergence, i.e., that inherited from $C( [0,\infty) \times \bbC \times \bbC, \Mat(2,\bbC) )$.

Let us now turn to continuity properties of the de Branges bijection between trace-normalized Hamiltonians and $m$-functions. 
\begin{theorem}\label{thm:Contin}
	Let $H,H_n$ be trace normalized Hamiltonians on $[0,\infty)$; $m,m_n$ the corresponding $m$-functions; $M_n,M$ the fundamental solutions; and $\{\cK^n_L\}_{L\geq0},\{\cK_L\}_{L\geq0}$ the associated kernels. 
Then the following conditions are equivalent:
	\begin{enumerate}[(i)]
		\item \label{it:ContinMC+}
		$m_n \to m$ in $\fM$;
		\item \label{it:ContinHamiltonian}
		$H_n \to H$ in $\fH$;
		\item \label{it:ContinFundamental}
		$M_n \to M$ in $\fS$;
		\item \label{it:ContinKernel}
		$\cK^n \to \cK$ in $\fK$.
	\end{enumerate}
\end{theorem}

\begin{proof}
The equivalence of \eqref{it:ContinHamiltonian}, \eqref{it:ContinFundamental}, and \eqref{it:ContinMC+} is well known and presented already in \cite{EckKostTeschl}, see also \cite[Chapter 5]{RemlingBookCanSys}.  Some of these references formulate (i), (ii) with pointwise convergence, but also note that this corresponds to the same topology by applications of the Arzel\`a--Ascoli theorem. Thus, it remains to show the equivalence of \eqref{it:ContinFundamental} and \eqref{it:ContinKernel}.

\eqref{it:ContinFundamental} $\implies$ \eqref{it:ContinKernel}:  It suffices to prove uniform convergence
\begin{equation} \label{locallyuniformconvcK}
\sup_{(t,z,w) \in S} \lvert \cK^n_t(z, \overline w) - \cK_t(z,\overline w) \rvert \to 0, \qquad n\to\infty
\end{equation}
on compacts of the form $S = [0,T] \times S_1 \times S_2 \subset [0,\infty) \times \bbC \times \bbC$.

Consider first a special case: if $S_1\cap S_2=\emptyset$, \eqref{locallyuniformconvcK} follows immediately from
\[
\cK^n_t(z,\overline w)=\frac{M_n(t,\overline{w})^*jM_n(t,z)-j}{ w-z}
\]
by locally uniform convergence  $M_n \to M$.

Now assume that $S_1\cap S_2\neq \emptyset$. Choose $R>0$ such that $S_1 \cup S_2 \subset B_R(0)$ and consider $\tilde S_1=\partial B_{2R}(0)$.  Since the kernels are entire functions of $z$, we obtain by the maximum principle that 
\[
\sup_{z \in S_1}\left|\cK^n_t(z,\overline w)-\cK_t(z,\overline w)\right| \le \sup_{z \in \tilde S_1}\left|\cK^n_t(z,\overline w)-\cK_t(z,\overline w)\right|
\]
Apply to this inequality $\sup_{t\in [0,T]} \sup_{w \in S_2}$. Then the right-hand side goes to $0$ as $n\to\infty$ by the special case noted above, because $\tilde S_1 \cap S_2 = \emptyset$. Thus, \eqref{locallyuniformconvcK} holds in general.

 \eqref{it:ContinKernel} $\implies$ \eqref{it:ContinFundamental}: This follows from $M_n(t,z)=I_2+zj\cK^n_t(z,0)$. 
\end{proof}

In other words, the correspondences between $\fM,\fH, \fS, \fK$ are homeomorphisms. Since $\fM$ is a compact metric space, so are $\fH, \fS, \fK$.

\section{Universality from normal limits of $m$} \label{sectionLimitm}

In this section we prove Theorem~\ref{thmuniformnormallimit} and its orthogonal polynomial consequences. The key will be to consider a family of trace-normalized canonical systems parametrized by $(\xi, r) \in J \times [1,\infty)$ and investigate  whether  this family has a continuous extension to $J \times [1,\infty]$. Here $[1,\infty]$ is used with the metric inherited from, say, the Riemann sphere $\hat\bbC$; it is used as a compact interval.  An abstract metric space argument relates existence of a continuous extension to uniform convergence:

\begin{lemma}
Let $F: J \times [1,\infty] \to Y$ be a map to some metric space $Y$, such that the restrictions $F\rvert_{ J \times [1,\infty) }$ and  $F\rvert_{ J \times \{\infty\} }$ are continuous. Then $F$ is continuous if and only if $F(\xi, r) \to F(\xi, \infty)$ as $r \to \infty$ uniformly in $\xi$.
\end{lemma}

\begin{proof}
If $F$ is continuous, then it is uniformly continuous on the compact (in $\bbR\times\hat{\bbC}$) $J \times [1,\infty]$. This uniform continuity implies uniform convergence $F(\cdot, r) \to F(\cdot, \infty)$ as $r \to \infty$.

For the converse, it suffices to prove continuity of $F$ at points $(\xi, \infty)$  by proving convergence along sequences $(\xi_n, r_n) \to (\xi,\infty)$. Note
\[
d_Y ( F(\xi_n ,r_n), F(\xi,\infty) ) \leq d_Y ( F(\xi_n,r_n), F(\xi_n,\infty)) + d_Y ( F(\xi_n,\infty), F(\xi,\infty) )
\]
The first term is small for large $n$ by uniform convergence (even if $r_n=\infty$) and the second is small for large $n$ by continuity on $J \times \{\infty\}$. 
\end{proof}

The family of canonical systems which we wish to consider can be described by their $m$-functions. Starting from a function $m\in \fM$ (which will later be chosen to be the function from the statement of Theorem~\ref{thmuniformnormallimit}), we define 
\begin{equation}\label{8jul1}
m_{\xi,r}(z) = m(\xi + z/r), \qquad (\xi,r) \in J \times [1,\infty)
\end{equation}
and view this as a map $J \times [1,\infty) \to \fM$. The first observation is that if this map has a continuous extension, that extension inevitably has constant functions at $r =\infty$:

\begin{lemma} \label{lemmanormallimitconstant}
Fix $m \in \fM$ and $\xi \in \bbR$ and consider the functions $m_r(z) = m( \xi + z / r)$ for $r \in [1,\infty)$. If the functions $m_r$ converge in $\fM$ as $r\to \infty$, the limit is a constant function, i.e., $m_r \to \eta$ for some $\eta \in \overline{\bbC_+}$.
\end{lemma}

\begin{proof}
Without loss of generality assume $\xi = 0$. Denote by $f \in \fM$ the limit of $m_{\xi,r}$ as $r \to \infty$. Then for any $y > 0$,
\[
f(iy) = \lim_{r\to\infty} m(iy/r) = \lim_{\epsilon \downarrow 0} m(i\epsilon),
\]
by the change of variables $\epsilon = y/r$. This implies that $f(iy)$ is independent of $y$, so $f$ is constant.
\end{proof}

We note that Lemma~\ref{lemmanormallimitconstant} would not be true sequentially: there exist $m\in \fM$ and sequences $r_n \to \infty$ such that $m_{r_n}$ converges in $\fM$ to a nonconstant Herglotz function. For example, using the branch of $\ln$ on the right half-plane with $\lvert \Im \ln \rvert < \pi/ 2$, we can define for $z\in \bbC_+$
\[
m(z) = i \exp\left( \frac{\pi/2}{e^{\pi/2}} \sin \ln (-iz) \right).
\]
This is a Herglotz function because $\sin$ maps the region $\lvert \Im  \rvert < \pi/2$ into the region $\lvert \Im \rvert < e^{\pi/2}$. Moreover, since $m(z/ e^{2\pi}) = m(z)$, the family of rescalings $m(z/r)$ is compact and indexed by $\log r \in \bbR / (2\pi \bbZ)$. Thus, any subsequential limit of $m(\cdot / r_n)$ in $\fM$ with $r_n \to \infty$ is of the form $m(z / r)$ for some $r \in [1,e^{2\pi})$.

Next, we interpret the existence of a continuous extension for the family \eqref{8jul1} in terms of normal limits and nontangential limits of $m$:

\begin{lemma} 
Let $m\in \fM$ and let $\eta: J \to \overline{\bbC_+}$ be a continuous function. The following are equivalent:
\begin{enumerate}[(a)]
\item The map $J \times [1,\infty] \to \fM$ given by 
\begin{equation} \label{mfunctionscontfamily}
m_{\xi,r}(z) = \begin{cases} m( \xi + z/r) & r \in [1,\infty) \\
\eta(\xi) & r =\infty
\end{cases}
\end{equation}
 is continuous (with respect to the metric on $\fM$).
\item (uniform nontangential limits) For each $\alpha \in (0,\pi/2]$,  uniformly in $\xi \in J$, $m(\xi+z) \to \eta(\xi)$ as $z \to 0$ with $\alpha \le \arg z \le \pi - \alpha$.
\item (uniform normal limits) $m(\xi+iy) \to\eta(\xi)$ as $y\downarrow 0$, uniformly in $\xi \in J$.
\end{enumerate}
\end{lemma}

\begin{proof}
(a) $\implies$ (b): For any $\alpha > 0$, the set $S = \{ \zeta = e^{i\phi} \mid  \phi \in [\alpha,\pi - \alpha] \} \subset \bbC_+$ is compact. By uniform continuity of $m_{\xi,r}(\zeta)$ in $\zeta \in S$, letting $r \to \infty$ with $z = \zeta / r$, it follows with $m(\xi + z) \to \eta(\xi)$ uniformly in $\xi$.

(b) $\implies$ (c): is trivial, by taking $\alpha = \pi/2$.

(c) $\implies$ (a): it is obvious that $m_{\xi,r}$ is continuous on $J \times [1,\infty)$ and on $J \times \{\infty\}$. Thus, it suffices to prove convergence along sequences $(\xi_n, r_n) \to (\xi, \infty)$ with $r_n \neq \infty$ for all $n$.

Assume first that such a sequence is chosen so that $m_{\xi_n ,r_n} \to f \in \fM$. Evaluating, in particular, at $\zeta = iy$, $y > 0$, we get
\[
m(\xi_n + iy / r_n) \to f(iy), \quad n\to\infty.
\]
Meanwhile, (c) implies that $m(\xi_n + iy / r_n) \to \eta(\xi)$, so we conclude $f(iy) = \eta(\xi)$ for all $y > 0$. This determines uniquely the Herglotz function $f$ as $f \equiv \eta(\xi)$. Thus, the constant Herglotz function $\eta(\xi)$ is the unique accumulation point along sequences $(\xi_n, r_n) \to (\xi, \infty)$ with $r_n \neq \infty$ for all $n$; by compactness of $\fM$, this shows that $m_{\xi_n, r_n} \to \eta(\xi)$ in $\fM$ as $n \to\infty$, which gives the desired continuity.
\end{proof}

Let us consider trace-normalized canonical systems corresponding to $m_{\xi,r}$ and denote the Hamiltonians, solutions, matrix CD kernels, and scalar CD kernels by  $ H_{\xi,r},  M_{\xi,r},  \cK_{\xi,r},  K_{\xi,r}$, respectively. Since $m_{\xi,1}$ is constructed from $m$ by a shift of the spectral parameter, combining the shift by $\xi$ with the gauge transformation \eqref{3jul3} and a trace-parametrization (Lemma~\ref{lem:Reparam}), it follows that
\begin{align}
\int_0^t H_{\xi,1}(\tau)\,d\tau & = \cK_L(\xi,\xi) \label{27jul1a} \\
M_{\xi,1} (t,z) & = T(L,\xi)^{-1} T(L,\xi+ z) \label{27jul1b} \\
 (\cK_{\xi,1})_t(z,w) & = \cK_L(\xi+z,\xi+w) \label{27jul1c}
\end{align}
where  the trace-normalized parameter $t$ is a monotone increasing function of $L$. By taking traces of both sides of \eqref{27jul1a}, we see that 
\[
t = \tau_\xi(L),
\]
with $\tau_\xi(L) = \cK_L(\xi,\xi)$ as defined by \eqref{universalityscale}.

Next, we note how scaling in $r \in [1,\infty)$ affects the canonical systems:

\begin{lemma}\label{lem:rInvariance}
The effect of the scaling by $r$ in \eqref{8jul1} is as follows: the Hamiltonian scales as
\begin{equation}\label{Hxirstretch}
H_{\xi,r}(t)=H_{\xi,1}(rt),
\end{equation}
the fundamental solution is given by
\begin{equation}\label{Mxirstretch}
M_{\xi,r}(t,z)=M_{\xi,1}(rt,z/r),
\end{equation}
and the corresponding kernel $(\cK_{\xi,r})_t$ is given by
\begin{equation}\label{Kxirstretch}
(\cK_{\xi,r})_{t}(z,w)=\frac{1}{r}(\cK_{\xi,1})_{rt}(z/r,w/r).
\end{equation}
\end{lemma}

\begin{proof}
If we had taken \eqref{Hxirstretch} as the definition for arbitrary $r$, the scaling \eqref{Mxirstretch} and the scaling
\[
m_{\xi,r}(z)=m_{\xi,1}(z/r),\quad z\in\bbC_+
\]
would follow by direct computation, see \cite[Lemma 2.7]{EckKostTeschl}. In particular, this matches our definition \eqref{8jul1}, so by bijectivity of the correspondences, this proves \eqref{Hxirstretch} and \eqref{Mxirstretch}. We also have 
\begin{align*}
	(\cK_{\xi,r})_{t}(z,w) &=\frac{M_{\xi,r}(t,w)^*jM_{\xi,r}(t,z)-j}{\overline{w}-z}\\
	&=\frac{1}{r}\frac{M_{\xi,1}(rt,w/r)^*jM_{\xi,1}(rt,z/r)-j}{\overline{w/r}-z/r}=\frac{1}{r} (\cK_{\xi,1})_{rt}(z/r,w/r). \qedhere
\end{align*}
\end{proof}

Due to the homeomorphisms provided by Theorem~\ref{thm:Contin}, the existence of a continuous extension of the family $\{ m_{\xi,r} \}$ can equivalently be checked at the level of the associated objects $H_{\xi,r}, M_{\xi,r}, \cK_{\xi,r}$. Moreover, by Lemma \ref{lemmanormallimitconstant} any continuous extension must correspond to constant Hamiltonian canonical systems at $r = \infty$.  

If $m_{\xi,\infty}(z) = \eta(\xi)$, then applying scaling formulas to the constant Hamiltonian canonical systems from Example~\ref{xmplconsthamiltonian} and noting that $\mathring H, \mathring M, \mathring \cK$ correspond to evaluation at $t=1$, we obtain 
\[
H_{\xi,\infty} (t) = \mathring H_{\eta(\xi)}, \quad M_{\xi,\infty}(t,z) = \mathring M_{\eta(\xi)}(tz), \quad (\cK_{\xi,\infty})_t(z,w) = t \mathring \cK_{\eta(\xi)}(tz,tw).
\]

The following lemma characterizes the possibility of continuous extensions in terms of Hamiltonians: 

\begin{lemma}
The family  of trace-normalized Hamiltonians $H_{\xi,r}$ parametrized by $(\xi,r) \in J \times [1,\infty)$ has a continuous extension to $J \times [1,\infty]$ with constant Hamiltonians $H_{\xi,\infty}(t) = \mathring H_{\eta(\xi)}$ if and only if
\begin{equation} \label{8jul4}
\lim_{\tau \to \infty} \frac 1\tau \int_0^\tau H_{\xi,1}(t) dt = \mathring H_{\eta(\xi)}
\end{equation}
uniformly in $\xi \in J$.
\end{lemma}

\begin{proof}
For any $T \ge 0$, by the formula \eqref{Hxirstretch},
\begin{equation}\label{8jul5}
\frac 1T \int_0^T H_{\xi, r}(t) dt = \frac 1T \int_0^T H_{\xi, 1}(rt) dt = \frac 1{rT} \int_0^{rT} H_{\xi, 1}(t) dt. 
\end{equation}
If $H_{\xi,r} \to \mathring H_{\eta(\xi)}$ in $\fH$ as $r \to \infty$, using $T = 1$ and $r = \tau  \to \infty$ in \eqref{8jul5} gives the uniform convergence in \eqref{8jul4}. Conversely, if \eqref{8jul4} holds uniformly in $\xi$, then for any $T > 0$, the integral \eqref{8jul5} extends continuously to $(\xi,r) \in J \times [1,\infty]$; by compactness of $\fH$, this implies continuity of the map $J \times [1,\infty] \to \fH$.
\end{proof}

By similar arguments, the scalings \eqref{Mxirstretch}, \eqref{Kxirstretch} imply the following lemmas:

\begin{lemma}
This family of solutions of canonical systems $M_{\xi,r}(t,z)$ has a continuous extension to a map $J \times [1,\infty] \to \fS$ with functions $M_{\xi,\infty}(t,z) = \mathring M_{\eta(\xi)}(tz)$ if and only if
\begin{equation}\label{8jul7}
\lim_{\tau \to \infty}  M_{\xi,1} (\tau , z/\tau) = \mathring M_{\eta(\xi)}(z)
\end{equation}
uniformly on compact subsets of $(\xi, z) \in J \times \bbC$.
\end{lemma}

\begin{proof}
If $M_{\xi,r} \to M_{\xi,\infty}$ in $\fS$ uniformly in $\xi$, evaluating at $t = 1$ and using $\tau = r$ gives \eqref{8jul7}. Conversely, if \eqref{8jul7} holds, then for each $t >0$, $z \in \bbC$, the function
\[
M_{\xi,r}(t,z) = M_{\xi,1}(rt, z/ r) = M_{\xi,1} ( rt , (tz) / (rt) )
\]
extends continuously to $J \times [1,\infty]$ with values $M_{\xi,\infty} (t,z) = \mathring M_{\eta(\xi)}(tz)$. Thus, by compactness of $\fS$, $M_{\xi,r}$ is continuous on $J \times [1,\infty]$.
\end{proof}

\begin{lemma} \label{lemma47}
This family of kernels $\cK_{\xi,r}$ has a continuous extension to a map $J \times [1,\infty] \to \fK$ with functions $(\cK_{\xi,\infty})_t(z,w) = t \mathring \cK_{\eta(\xi)}(tz,tw)$ if and only if
\begin{equation}\label{8jul9}
\lim_{\tau \to \infty}  \frac 1\tau (\cK_{\xi,1})_\tau (z/\tau, w/\tau) = \mathring \cK_{\eta(\xi)}(z,w)
\end{equation}
uniformly on compact subsets of $(\xi,z,w) \in J \times \bbC \times \bbC$.
\end{lemma}

\begin{proof}
If $\cK_{\xi,r} \to \cK_{\xi,\infty}$ in $\fK$ uniformly in $\xi$, evaluating these kernels at $t = 1$ and using $\tau = r$ gives \eqref{8jul9}. Conversely, if \eqref{8jul9} holds, then for each $t >0$, $z, w \in \bbC$, the function
\[
(\cK_{\xi,r})_t(z,w) = \frac 1r (\cK_{\xi,1})_{rt}(z/r,w/r)  = \frac t{rt} (\cK_{\xi,1})_{rt} \left(\frac{tz}{rt},\frac{tw}{rt} \right) 
\]
extends continuously to $J \times [1,\infty]$ with values $(\cK_{\xi,\infty})_t(z,w)  = t \mathring \cK_{\eta(\xi)}(tz,tw)$. Thus, by compactness of $\fS$, $\cK_{\xi,r}$ is continuous on $J \times [1,\infty]$.
\end{proof}

Until now, everything was derived from a starting function $m \in \fM$. By taking that to be the Weyl function in the setting of Hypothesis~\ref{Txzcanonicalsystem} and linking the objects $T(L,z)$, $\cK_L(z,w)$ to the continuous families indexed by $(\xi,r)$ above, we will obtain the proof of Theorem~\ref{thmuniformnormallimit}.

\begin{proof}[Proof of Theorem~\ref{thmuniformnormallimit}]
Combining the lemmas above, we conclude that the following statements are mutually equivalent:
\begin{enumerate}[(i')]
\item $m(\xi + i y) \to \eta(\xi)$ as $y \downarrow 0$, uniformly in $\xi \in J$;
\item there exists continuous $\eta:J \to \overline{\bbC_+}$ such that \eqref{8jul4} holds uniformly in $\xi \in J$;
\item there exists continuous $\eta:J \to \overline{\bbC_+}$ such that \eqref{8jul7} holds uniformly in $\xi \in J$;
\item there exists continuous $\eta:J \to \overline{\bbC_+}$ such that \eqref{8jul9} holds uniformly in $\xi \in J$.
\end{enumerate}
Moreover, the function $\eta$ then must be the same in all four statements.  Clearly, (i') is precisely Theorem~\ref{thmuniformnormallimit}(i). Let us compare the other statements to those in Theorem~\ref{thmuniformnormallimit}.

Since the family is in the limit point case, for each $\xi \in J$,
\[
\lim_{L \to\infty} \tau_\xi(L) = \infty.
\]
Moreover, $\tau_\xi$ is monotone increasing in $L$ for each $\xi$, so convergence is also uniform on compact intervals by Dini's theorem,
\begin{equation}\label{Dini1}
\lim_{L\to\infty} \inf_{\xi \in J} \tau_\xi(L) = \infty.
\end{equation}
Thus, in the reparametrization from $(\xi,L)$ to $(\xi, t) = (\xi, \tau_\xi(L))$, uniform convergence statements as $L \to \infty$ are equivalent to uniform convergence statements as $t \to \infty$.

Using \eqref{27jul1a}, \eqref{27jul1b}, \eqref{27jul1c} with $t = \tau_\xi(L)$ given by \eqref{universalityscale}, statements (ii'), (iii'), (iv') above turn into statements (ii), (iii), (iv) of Theorem~\ref{thmuniformnormallimit}, and the proof is complete.
\end{proof}

\begin{proof}[Proof of Theorem~\ref{theoremJacobiEquivalenceUniform}]
This is mostly just a special case of Theorem~\ref{thmuniformnormallimit}. The only difference is in condition (ii), which is stated here over the sequence indexed by $n$. Obviously, if 
\begin{equation}\label{25jul1}
\lim_{L\to\infty} \frac 1{\tau_\xi(L)} \cK_L(\xi,\xi) = H,
\end{equation}
then 
\begin{equation}\label{25jul2}
\lim_{n\to\infty} \frac 1{\tau_\xi(n)} \cK_n(\xi,\xi) = H.
\end{equation}
Conversely, using linearity of $\cK_L(\xi,\xi)$ in $L \in [n,n+1]$ for each $n$,
\[
\frac 1{\tau_\xi(L)} \cK_L(\xi,\xi)  = s  \frac1 {\tau_\xi(n+1)} \cK_{n+1}(\xi,\xi) + (1-s)  \frac1 {\tau_\xi(n)} \cK_{n}(\xi,\xi)
\]
where
\[
s = \frac{\tau_\xi(n+1) (L-n)}{ \tau_\xi(n+1) (L-n) + \tau_\xi(n) (n+1 - L)} \in [0,1],
\]
so these are convex combinations of $\frac 1{\tau_\xi(n)} \cK_n(\xi,\xi)$.  Thus, \eqref{25jul2} implies \eqref{25jul1}.
\end{proof}

\begin{proof}[Proof of Theorem~\ref{thmlimitmatrixCDkernel2}]
This follows from the implication (i)$\implies$(iv) in Theorem~\ref{theoremJacobiEquivalenceUniform}.
\end{proof}

\section{Universality from nontangential limits of $\Im m$: rescaling and shifting} \label{sectionLimitImm}

When we only have at our disposal the limit behavior of $\Im m$, the situation is more complicated. The following lemma  reinterprets nontangential boundary limits of $\Im m$ in terms of continuity in $\fM$, by showing that we can correct the behavior of $\Re m$ as $r \to \infty$ by an additive action.

\begin{lemma} \label{lemmaCaratheodory}
Assume that for some $0 < \alpha < \pi/2$,  \eqref{limitImm} holds uniformly in $\xi \in J$ and denote
\[
c_{\xi,r} = \Re m(\xi + i /r), \qquad (\xi,r) \in J \times [1,\infty).
\]
Then $m_{\xi,r} - c_{\xi,r} \to i \pi f_\mu(\xi)$ in $\fM$ as $r \to\infty$, uniformly in $\xi \in J$.
\end{lemma}

\begin{proof}
Precomposing all the functions
\[
\tilde m_{\xi,r} = m_{\xi,r} - c_{\xi,r}
\]
with the Cayley transform $\gamma : \bbD  \to \bbC_+$ and multiplying by $-i$ gives analytic maps
\[
g_{\xi,r} =  - i \tilde m_{\xi,r} \circ \gamma : \bbD \to \{w \in \bbC \mid \Re w > 0\}.
\]
Fix some radii $0 < R_1 < R_2 < 1$ such that the image $\gamma(\overline{\bbD_{R_2}(0)})$ is contained in the sector $\{ z \mid \alpha \le \arg z \le \pi - \alpha \}$. The complex Poisson representation (Schwarz integral formula) applied to the circle of radius $R_2$ around $0$ gives
\begin{equation}\label{Schwarz}
g_{\xi,r}(\z) = i \Im g_{\xi,r}(0) + \int_0^{2\pi}\frac{R_2 e^{i\phi} + \z}{R_2 e^{i\phi} - \z} \Re g_{\xi,r}(R_2 e^{i\phi}) \frac{d\phi}{2\pi}.
\end{equation}
Note that $\Im g_{\xi,r}(0) = \Re m_{\xi,r}(i) - c_{\xi,r}  = 0$. The kernels obey
\[
\sup_{\zeta \in \overline{\bbD_{R_1}(0)} }  \int_0^{2\pi} \left\lvert \frac{R_2 e^{i\phi} + \z}{R_2 e^{i\phi} - \z} \right\rvert \frac{d\phi}{2\pi}< \infty
\]
so for $(\xi_n, r_n) \to (\xi, r)$,  uniform convergence $\Re g_{\xi_n,r_n} \to \Re g_{\xi,r}$ on the circle $\lvert \zeta \rvert = R_2$ implies uniform convergence  $g_{\xi_n,r_n} \to  g_{\xi,r}$ on the disk $\overline{\bbD_{R_1}(0)}$ and therefore uniform convergence $m_{\xi_n, r_n} \to m_{\xi,r}$ on the compact $\gamma( \overline{\bbD_{R_1}(0)} )$. By compactness of $\fM$, this implies locally uniform convergence $m_{\xi_n, r_n} \to m_{\xi,r}$ on $\bbC_+$. 
\end{proof}

Lemma~\ref{lemmaCaratheodory} paves the way for the following modification of the strategy from Section~\ref{sectionLimitm}: let us consider the family of $m$-functions indexed by $(\xi,r) \in J \times [1,\infty]$ given by
\begin{equation}\label{tildemxir}
\tilde m_{\xi,r}(z) = \begin{cases}
m(\xi + z/r) - c_{\xi,r} & r \in [1,\infty) \\
i \pi  f_\mu(\xi) & r = \infty
\end{cases}
\end{equation}
By Lemma~\ref{lemmaCaratheodory} this is a continuous family. We also consider the corresponding trace normalized canonical systems, denoting the Hamiltonians, solutions, matrix CD kernels, and scalar CD kernels by  $\tilde H_{\xi,r}, \tilde M_{\xi,r}, \tilde \cK_{\xi,r}, \tilde K_{\xi,r}$, respectively.

The necessary additive shift to the Weyl function can be implemented by action by a triangular matrix. The action by a matrix $A \in \SL(2,\bbR)$ at the level of the $m$-function by
\[
m_A(z) = A^{-1} m(z)
\]
(recall that we use $A^{-1}$ also for the generated M\"obius transformation) corresponds to
\[
 H_A = A^* H A, \quad M_A = A^{-1} M A, \quad \cK_A = A^* \cK A
\]
(this is a combination of several lemmas in \cite{RemlingBookCanSys} and a direct calculation for the kernel). We emphasize that this does not preserve trace normalization. The translations like those in \eqref{tildemxir} are encoded by matrices of the form
\[
A  =  \begin{pmatrix}
1 & a \\
0 & 1
\end{pmatrix},\quad a\in\bbR,
\]
and we emphasize the observation that for such $A$, $A \binom 10 = \binom 10$ implies
\[
{ \binom 10}^* \cK_A \binom 10 = { \binom 10}^* \cK \binom 10.
\]
This is a crucial observation: it tells us that the scalar CD kernels are unaffected by $A$, except for the fact that trace-normalization changes the parametrization. Although it is possible to reparametrize explicitly, this reparametrization is nonlinear and impractical. We will instead deal with this more implicitly and formulate the conclusion this way: as families of functions of $(z,w) \in \bbC \times \bbC$,
\[
\{ (\tilde K_{\xi,r} )_t (z,w) \mid t \in [0,\infty ) \}  = \{ ( K_{\xi,r} )_t (z,w) \mid t \in [0,\infty ) \} 
\]
 for any $(\xi,r) \in J \times [1,\infty)$. Combining this with the previous section shows that
\begin{equation}\label{22jul3}
\left \{  \frac 1{r} K_L\left( \xi + \frac z{r} , \xi + \frac w{r} \right)  \mid L \in [0,\infty) \right \} = \left\{ ( \tilde K_{\xi,r})_t(z,w)  \mid t \in [0,\infty) \right\}. 
 \end{equation}

 Since we wish to work at scale $K_L(\xi,\xi)$, it is important to observe that that scale is unbounded. For orthogonal polynomials, it is a well-known fact that $\sum_{j=0}^\infty p_j(\xi)^2 = \infty$ corresponds to a lack of point mass in the measure and is equivalent to
 \begin{equation} \label{eqnlackofpointmass}
\lim\limits_{\e\to 0}\e \Im m(\xi+i\e)=0.
\end{equation}
The same holds for canonical systems:

\begin{lemma} \label{lemmalackofpointmass}
Let $m: \bbC_+ \to \bbC_+$ be the Weyl function of a canonical system. For any $\xi\in\bbR$, \eqref{eqnlackofpointmass} is equivalent to
\begin{align}
\lim\limits_{L\to\infty}K_L(\xi,\xi)=\infty.	
\end{align}
\end{lemma}

\begin{proof}
We first note that the claim is translation invariant, so it suffices to prove the case $\xi=0$, and that it is gauge-invariant, so it suffices to prove the case of canonical systems.

Since $m\neq\infty$, there are uniquely determined $a\geq 0,b\in\bbR$ and a measure $\mu$ such that \[
m(z)=az+b+\int_{-\infty}^\infty\left(\frac{1}{\xi-z}-\frac{\xi}{1+\xi^2}\right)d\mu(\xi), \qquad \int \frac 1{1+\xi^2} d\mu(\xi) < \infty.
\]
Moreover, $\lim_{\e\to 0}\e \Im m(i\e) = \mu( \{0\})$. This limit is zero if and only if the operator of multiplication by $\xi$ in $L^2(\bbR, d\mu(\xi))$ does not have an eigenvalue at $0$.

Denote the Hilbert space
\[
L^2_H=\left\{f:(0,\infty)\to \bbC^2: \int_0^\infty f^*(x)H(x)f(x)dx<\infty \right\}.
\]
The canonical system relation, with a Dirichlet boundary condition at $0$, is the linear subspace $\cT \subset L^2_H \oplus L^2_H$ defined by
\[
\cT = \left\{ (f,g) \in L^2_H \oplus L^2_H \mid  f(T) -f(0) = j  \int_0^T H(t) g(t) \,dt \; \forall T \in [0,\infty), \; \begin{pmatrix} 0 & 1 \end{pmatrix} f(0) = 0\right\}.
\]
Its kernel is the set of $(f,0) \in \cT$; elements of the kernel can only be scalar multiples of $f(t) = e_1 = \begin{pmatrix}
	1&0
	\end{pmatrix}^*$. To check whether these lie in the Hilbert space, note that 
\[
K_L(0,0)=\int_0^Le_1^*H(x)e_1dx,
\]
so $\cT$ has a nontrivial kernel if and only if $\lim_{L\to\infty} K_L(0,0) < \infty$. This concludes the proof.
\end{proof}

The discussion so far indicates that instead of $t$, we would prefer to parametrize the kernels by 
\[
\sigma_{\xi,r}(t) = (\tilde K_{\xi,r})_t(0,0).
\]
By Lemma~\ref{lemmalackofpointmass} and our assumption on $m$, the functions $\tilde m_{\xi,r}$ correspond to measures without a point mass at $0$. Thus, for each $\xi, r$, viewed as a function of $t$, $\sigma_{\xi,r} : [0,\infty) \to [0,\infty)$ is an increasing surjection. By Dini's theorem,
\begin{equation}\label{Dini3}
\lim_{t\to\infty} \inf_{(\xi,r) \in J \times [1,\infty]} \sigma_{\xi,r}(t) = \infty.
\end{equation}
Moreover, $\sigma_{\xi,r}(t)$ is jointly continuous in $\xi,r,t$, since $(\tilde \cK_{\xi,r})_t$ is.

However, $\sigma_{\xi,r}$ may not be injective, so it is not a good candidate for a parameter. This is related to the notion of singular intervals; in the orthogonal polynomial case this noninjectivity is connected to   the possibility that $p_n(\xi) = 0$. Thus, we will not reparametrize by $s = \sigma_{\xi,r}(t)$, but we will study the set on which $ \sigma_{\xi,r}(t) = 1$ and proceed more indirectly:

\begin{lemma} \label{lemmacS1}
The set
\[
\cS = \{ (\xi,r,t) \in J \times [1,\infty] \times [0,\infty) \mid \sigma_{\xi,r}(t) = 1 \}
\]
is a compact subset of $J \times [1,\infty] \times [0,\infty)$.
\end{lemma}

\begin{proof}
Since $\sigma_{\xi,r}(t)$ is jointly continuous, the set $\cS$ is closed as the inverse image of $\{1\}$. By \eqref{Dini3}, the set $\cS$ is contained in some compact of the form $J \times [1,\infty] \times [0,T]$, $T < \infty$. Thus, as a closed subset of a compact set, $\cS$ is compact.
\end{proof}

Another necessary observation is that $\sigma_{\xi,r}$ are bijections for $r = \infty$: 

\begin{lemma}\label{lemmacS2}
For any $\xi \in J$, there is a unique $t$ such that $(\xi,\infty,t) \in \cS$, and it is given by
\[
t = 1 +\pi^2 f_\mu(\xi)^2.
\]
\end{lemma}

\begin{proof}
Starting from 
\[
(\tilde \cK_{\xi,\infty})_t (z,w) = t \mathring \cK_{i\pi f_\mu(\xi)} (tz, tw) 
\]
a direct calculation gives
\begin{equation}\label{eqn6}
(\tilde K_{\xi,\infty})_t (z,w) = \frac{ \sin\left( \frac{ \pi f_\mu(\xi)}{1 + \pi^2 f_\mu(\xi)^2 } t (\overline{w} - z )\right) }{ \pi f_\mu(\xi) ( \overline{w} - z ) }
\end{equation}
and in particular
\[
(\tilde K_{\xi,\infty})_t (0,0) = \frac t{1+\pi^2 f_\mu(\xi)^2}. 
\]
From this the claim is obvious.
\end{proof}

Now we can collect these observations and find our universality limit in the set of kernels $(\tilde \cK_{\xi,r})_t$ with $(\xi,r,t) \in \cS$.

\begin{proof}[Proof of Theorem~\ref{thmImaginaryLimitCanonical}]
Let us set $r$ as a function of $\xi, L$ as
\[
r = r(\xi,L) = K_L(\xi,\xi).
\]
By Lemma~\ref{lemmalackofpointmass},  $r(\xi, L) \to \infty$ as $L \to \infty$. 

By \eqref{22jul3} there exists $t(\xi,L)$ such that
\[
\frac 1{r(\xi,L)} K_L\left( \xi + \frac z{r(\xi,L)} , \xi + \frac w{r(\xi,L)} \right)  = (\tilde K_{\xi,r(\xi,L)})_{t(\xi,L)}(z,w)
\]
(note that $t(\xi,L)$ is not necessarily unique). By comparing values at $z=w=0$ we see that
\[
\frac 1{r(\xi,L)} K_L(\xi,\xi) = (\tilde K_{\xi,r(\xi,L)} )_{t(\xi,L)} (0,0) = \sigma_{\xi,r(\xi,L)}(t(\xi,L))
\]
and since the left-hand side equals $1$, we conclude that
\[
(\xi, r(\xi,L), t(\xi,L)) \in \cS.
\]
The property $(\xi, r(\xi,L), t(\xi,L)) \in \cS$ together with $r(\xi,L) \to \infty$ and Lemmas~\ref{lemmacS1}, \ref{lemmacS2} implies that
\[
\lim_{L \to\infty} t(\xi,L) = 1 +\pi^2 f_\mu(\xi)^2
\]
so continuity of reproducing kernels gives
\begin{equation}\label{eqn7}
\lim_{L\to\infty} (\tilde K_{\xi,r(\xi,L)})_{t(\xi,L)}(z,w) =  (\tilde K_{\xi,\infty})_{1+\pi^2f_\mu(\xi)^2} (z, w).
\end{equation}
Inserting our definition of $t(\xi,L)$ and evaluating the right-hand side using \eqref{eqn6}, we have proved that
\begin{equation}\label{22jul5}
\lim_{L \to\infty} \frac 1{r(\xi,L)} K_L\left( \xi + \frac z{r(\xi,L)} , \xi + \frac w{r(\xi,L)} \right)  = \frac{ \sin(\pi f_\mu(\xi) (\overline{w} - z ) ) }{ \pi f_\mu(\xi) (\overline{w} - z ) }
\end{equation} 
and by our definition of $r(\xi,L)$, this is the same as saying that
\begin{equation}\label{22jul6}
\lim_{L \to\infty} \frac 1{K_L(\xi,\xi)} K_L\left( \xi + \frac z{K_L(\xi,\xi)} , \xi + \frac w{K_L(\xi,\xi)} \right)  = \frac{ \sin( \pi f_\mu(\xi) (\overline{w} - z ) ) }{ \pi f_\mu(\xi) (\overline{w} - z ) }
\end{equation}
From here, the theorem as formulated follows by a linear rescaling of $z,w$ by a factor of $f_\mu(\xi)$.
\end{proof}

\begin{proof}[Proof of Theorems~\ref{thmImaginaryLimitJacobi2}, \ref{thmImaginaryLimitSchrodinger}]
Given the reductions of Jacobi matrices and Schr\"odinger operators to the setting of Hamiltonian systems, these are direct corollaries of Theorem~\ref{thmImaginaryLimitCanonical}.
\end{proof}

\begin{proof}[Proof of Theorem~\ref{thmImaginaryLimitJacobi}] 
This is the special case $J = \{\xi\}$  of Theorem~\ref{thmImaginaryLimitJacobi2}.
\end{proof}

\section{Orthogonal polynomials on the unit circle} \label{sectionOPUC}

Orthogonal polynomials on the unit circle obey the Szeg\H o recursion, which gives rise to the sequence $\{\alpha_n \}_{n=0}^\infty$ of Verblunsky coefficients $\alpha_n \in \bbD$ associated to $\mu$. The recursion can be written in matrix form as
\[
S(n+1,z) = 
A(\alpha_n,z) S(n,z), \qquad A(\alpha_n,z) = \frac 1{ \sqrt{1 - \lvert \alpha_{n}\rvert^2}} \begin{pmatrix}
z & - \overline{ \alpha_{n}} \\
- \alpha_{n} z & 1
\end{pmatrix},\qquad S(0,z)= I_2.
\]
We denote by
\[
\cJ = \begin{pmatrix}
-1 & 0 \\
0 & 1
\end{pmatrix}
\]
the signature matrix corresponding to the unit disk. The crucial property of the Szeg\H o recursion is:

\begin{lemma} \label{lemmaSzego1step}
For any $\alpha_n \in \bbD$ and $z \neq 0$,
\[
\frac{ A(\alpha_n,z)^* \cJ A(\alpha_n, z)}{ \lvert z \rvert}  - \cJ = \begin{pmatrix} 1 - \lvert z \rvert  &  0 \\
0 & \lvert z \rvert^{-1} - 1
\end{pmatrix}
\]
which is a positive matrix for $0 < \lvert z \rvert \le 1$, and a negative matrix for  $\lvert z \rvert \ge 1$.
\end{lemma}

The proof is a straightforward calculation. Note that division by $\lvert z \rvert$ compensates for the fact that $\det A(\alpha_n,z) = z$; it effectively normalizes the matrices to unit determinant. The directions of inequalities in Lemma~\ref{lemmaSzego1step} may seem opposite from what is desired; this will be rectified by conjugating $A(\alpha_n,z)$ and $S(n,z)$ by the matrix
\[
j_1 = \begin{pmatrix}
0 & 1 \\
1 & 0
\end{pmatrix}.
\]

We also denote by
\[
\cC =  \begin{pmatrix}
1 & -i \\
1 & i
\end{pmatrix}
\]
the matrix corresponding to the Cayley transform. The Cayley transform allows us to switch between signature matrices by
\[
\frac 12 \cC^* \cJ \cC = -i j
\]
(the factor $1/2$ corresponds to $\det \cC = 2i$ and would vanish if we had normalized the determinant of $\cC$). The values of the Carath\'eodory function $F(z)$ are uniquely determined by
\[
 \begin{pmatrix} F(z) +1 \\ F(z) - 1\end{pmatrix}^* S(n,z)^* \cJ S(n,z) \begin{pmatrix} F(z) +1 \\ F(z) - 1\end{pmatrix} \le 0 \qquad \forall n\in\bbN,z\in\bbD
\]
and this condition can be written (note $j_1 \cJ j_1 = - \cJ$) as
\begin{equation}\label{20aug3}
 \begin{pmatrix} i F(z)  \\ 1  \end{pmatrix}^* \cC^*  j_1 S(n,z)^* j_1 \cJ j_1 S(n,z) j_1 \cC \begin{pmatrix} i F(z) \\ 1\end{pmatrix} \ge 0 \qquad \forall n\in\bbN,z\in\bbD.
\end{equation}
Now we are ready to transform and embed the Szeg\H o transfer matrices into a canonical system. This embedding is inspired by a method used by Damanik--Yuditskii \cite{DamYud}, where comb domains related to OPUC were mapped to comb domains periodic in the spectral parameter.

\begin{lemma}
The functions
\[
T(n,z) = e^{-inz/2} \cC^{-1} j_1 S(n, e^{iz} ) j_1 \cC
\]
are a $j$-monotonic family of entire $j$-inner functions and obey $T(0,z)=I_2$. They have unit determinant and are in the limit point case with $m$-function
\[
m(z) = i F(e^{iz}).
\]
\end{lemma}

\begin{proof}
In terms of $T(n,z)$, Szeg\H o recursion rewrites as $T(0,z) = I_2$, 
\[
T(n+1,z) = L_n(z) T(n,z), \qquad L_n(z) = e^{-iz/2} \cC^{-1} j_1 A(\alpha_n, e^{iz})  j_1  \cC.
\]
Now Lemma~\ref{lemmaSzego1step} implies that
\[
\frac{ L_n(z)^* j L_n(z) - j}i \le 0
\]
 for $\lvert e^{iz} \rvert \le 1$, with the sign of the inequality switched for $\lvert e^{iz} \rvert \ge 1$. Multiplying by $T(n,z)^*$ from the left and $T(n,z)$ from the right implies the $j$-monotonicity.
 
The scalar factor $e^{-inz/2}$ compensates for $\det S(n,e^{iz}) = e^{inz}$ and ensures $\det T(n,z) = 1$.

Plugging in $e^{iz}$ in place of $z$ in \eqref{20aug3} and using $\cC^* \cJ \cC = -2ij$ gives
\[
 \begin{pmatrix} i F(e^{iz})  \\ 1  \end{pmatrix}^* T(n,z)^* (-ij)  T(n,z) \begin{pmatrix} i F(e^{iz}) \\ 1\end{pmatrix} \ge 0 \qquad \forall n\in\bbN, 
 \]
Since OPUC are always in the limit point case, this uniquely determines the value $iF(e^{iz})$ and allows us to read off $m(z) = i F(e^{iz})$.
\end{proof}

To embed this in a canonical system, we note that a direct calculation gives
\[
L_n(0)^{-1} L_n(z) = \cC^{-1} j_1 \begin{pmatrix} e^{iz/2} & 0 \\ 0 & e^{-iz/2} \end{pmatrix} j_1 \cC = \cC^{-1} j_1 e^{-iz \cJ/2} j_1 \cC
\]
and therefore
\[
L_n(0)^{-1} L_n(z) = \cC^{-1} e^{iz \cJ / 2} \cC = e^{z j/2}.
\]
Thus the matrix functions $M(n,z) = T(n,0)^{-1} T(n,z)$ obey
\[
M(n+1,z) = T(n,0)^{-1} L_n(0)^{-1} L(n,z) T(n,z) = T(n,0)^{-1} e^{zj/2} T(n,0) M(n,z)
\]
and they are interpolated on $[n,n+1]$ by the canonical system
\[
\partial_x M(x,z) = z j H(x) M(x,z), \qquad H(x) = \frac 12 j^{-1} T(\lfloor x \rfloor,0)^{-1} j T(\lfloor x \rfloor, 0) = \frac 12 T(\lfloor x \rfloor,0)^* T(\lfloor x \rfloor, 0).
\]
Note that $H$ is once again a quadratic expression involving entries of the transfer matrix, which allows to recover results from subordinacy theory from OPUC. We omit further details and turn to the kernels in order to prove Theorem~\ref{thmOPUC}.

Now we have two kernels at our disposal: the Christoffel--Darboux kernel for OPUC given by \eqref{CDkernelOPUC1} and the scalar kernel obtained by the prescription for canonical systems \eqref{frommatrixtoscalarkernel},
\[
K_n(z,w) = \begin{pmatrix} 1 \\0 \end{pmatrix}^* \frac{ T(n,w)^* j T(n,z) - j}{\overline{w} - z} \begin{pmatrix} 1 \\0 \end{pmatrix}.
\]
The two are closely related:

\begin{lemma} For all $z, w\in \bbC$,
\begin{equation}\label{CDkernelOPUC2}
e^{-in (z-\overline{w})/2}  k_n(e^{iz}, e^{iw} ) =  \frac{2 i ( \overline{w}-z)}{ 1 - e^{i (z-\overline{w})}} K_n(z,w).
\end{equation}
\end{lemma}

\begin{proof}
In terms of orthonormal polynomials $\varphi_n$, second kind orthonormal polynomials $\psi_n$, and the reflected polynomials $\varphi_n^*$ and  $\psi_n^*(z) = z^n \overline{ \psi_n (1 / \overline{z} ) }$,  the Szeg\H o transfer matrices are of the form
\[
S(n,z) = \begin{pmatrix}
\frac 12 ( \varphi_n(z) + \psi_n(z) ) & \frac 12 ( \varphi_n(z) - \psi_n(z) )  \\
\frac 12 ( \varphi_n^*(z) - \psi_n^*(z) ) & \frac 12 ( \varphi_n^*(z) + \psi_n^*(z) ) 
\end{pmatrix}.
\]
A direct calculation gives
\[
T_n(z)  \begin{pmatrix} 1 \\0 \end{pmatrix} = e^{-inz/2} \cC^{-1}  \begin{pmatrix}  \varphi_n^*(e^{iz}) \\ \varphi_n(e^{iz}) \end{pmatrix},
\]
so using $(\cC^{-1})^* j \cC^{-1} = \frac i2 \cJ$, we obtain
\[
K_n(z,w) = - \frac{i}2 \frac{ e^{-in (z - \overline{w}) / 2} }{\overline{w} - z} \left(  \varphi_n^*(e^{iz}) \overline{ \varphi_n^* (e^{iw}) } -  \varphi_n(e^{iz}) \overline{ \varphi_n (e^{iw}) } \right).
\]
Comparing this with \eqref{CDkernelOPUC1} gives \eqref{CDkernelOPUC2}.
\end{proof}

\begin{proof}[Proof of Theorem~\ref{thmOPUC}]
A limit of $\Re F$ along a sector implies a limit of $\Im m (z) = \Re F(e^{iz})$ along a suitable sector. Thus we obtain for some $0 < \beta < \pi/2$,
\[
\lim_{\substack{z \to \xi \\ \beta \le \arg(z-\xi) \le \pi - \beta}} \Im m(z) = g_\mu(\xi).
\]
Comparing this to \eqref{limitImm} gives $f_\mu(\xi) = g_\mu(\xi) / \pi$. Moreover, by \eqref{CDkernelOPUC2},
\[
k_n(e^{i\xi}, e^{i\xi}) = 2 K_n(\xi, \xi).
\]
Now substituting $\xi + z / k_n(e^{i\xi}, e^{i\xi})$, $\xi + w / k_n(e^{i\xi}, e^{i\xi})$ in \eqref{CDkernelOPUC2} instead of $z,w$, we conclude that  the left-hand side of \eqref{OPUCconclusion} is equal to
\[
\lim_{n\to\infty} \frac{K_n ( \xi + \frac {z}{2\pi f_\mu(\xi) K_n(\xi,\xi)},  \xi + \frac {w}{2\pi f_\mu(\xi) K_n(\xi,\xi)}   )}{ K_n(\xi,\xi)}.
\]
Using the conclusion of Theorem~\ref{thmImaginaryLimitCanonical} and rescaling $z,w$ by a factor of $2\pi$ completes the proof.
\end{proof}

\providecommand{\MR}[1]{}
\providecommand{\bysame}{\leavevmode\hbox to3em{\hrulefill}\thinspace}
\providecommand{\MR}{\relax\ifhmode\unskip\space\fi MR }
\providecommand{\MRhref}[2]{%
  \href{http://www.ams.org/mathscinet-getitem?mr=#1}{#2}
}
\providecommand{\href}[2]{#2}

\end{document}